\documentclass[11pt,a4paper]{amsart}
\usepackage[pdftex]{graphicx}
\usepackage[pdftex]{xcolor}
\usepackage{enumerate}
\usepackage{ulem}
\usepackage{mathtools,  stmaryrd}
\usepackage{xparse} \DeclarePairedDelimiterX{\Iintv}[1]{\llbracket}{\rrbracket}{\iintvargs{#1}}
\NewDocumentCommand{\iintvargs}{>{\SplitArgument{1}{,}}m}
{\iintvargsaux#1} %
\NewDocumentCommand{\iintvargsaux}{mm} {#1\mkern1.5mu,\mkern1.5mu#2}

\makeatletter
\newtheorem*{rep@theorem}{\rep@title}
\newcommand{\newreptheorem}[2]{%
\newenvironment{rep#1}[1]{%
 \def\rep@title{#2 \ref{##1}}%
 \begin{rep@theorem}}%
 {\end{rep@theorem}}}
\makeatother

\definecolor{RedOrange}{cmyk}{ 0, 0.77, 0.87, 0}
\definecolor{RoyalPurple}{cmyk}{ 0.84, 0.53, 0, 0}
\definecolor{YellowGreen}{cmyk}{ 0.44, 0, 0.74, 0}
\definecolor{Fuchsia}{cmyk}{ 0.47, 0.91, 0, 0.08}
\definecolor{Blue}{cmyk}{ 0.84, 0.53, 0, 0}
\definecolor{BlueViolet}{cmyk}{ 0.84, 0.53, 0, 0}
\definecolor{Black}{cmyk}{ 0.75, 0.68, 0.67, 0.9}
\usepackage{verbatim}
\usepackage{amsmath}
\usepackage{amssymb, mathrsfs}
\usepackage{amsbsy}

\usepackage{amscd}
\usepackage{amsthm}

\usepackage[english]{babel}
\usepackage{todonotes}
\usepackage[T1]{fontenc}

\usepackage{color}

\newcommand{\lf}{\lfloor}
\newcommand{\rf}{\rfloor}
\newcommand{\R}{\mathbb{R}}

\newcommand{\N}{\mathbb{N}}
\newcommand{\e}{\epsilon}

\newcommand{\E}{\mathbb{E}}
\newcommand{\Z}{\mathbb{Z}}

\renewcommand{\O}{\mathbb{O}}
\renewcommand{\P}{\mathbb{P}}

\newcommand{\kE}{\mathcal{E}}

\newcommand{\rmT}{\mathrm{T}}

\newcommand{\rmC}{\mathrm{C}}

\newcommand{\lin}{\left[\kern-0.15em\left[}
\newcommand{\rin} {\right]\kern-0.15em\right]}
\newcommand{\linf}{[\kern-0.15em [}
\newcommand{\rinf} {]\kern-0.15em ]}
\newcommand{\ilin}{\left]\kern-0.15em\left]}
\newcommand{\irin} {\right[\kern-0.15em\right[}

\def\ben#1{\begin{equation}#1\end{equation}}

\def\al#1{\begin{align*}#1\end{align*}}
\def\aln#1{\begin{align}#1\end{align}}

\usepackage{constants}

\newconstantfamily{c}{symbol=c}

\newconstantfamily{a}{symbol=\alpha}

\newcommand{\secno}[1]{\thesection.\arabic{#1}}
\newconstantfamily{kE}{
symbol=\mathcal{E},
format=\secno,
reset={section}
}

\renewcommand{\tilde}{\widetilde}

\newtheorem{lem}{Lemma}[section]
\newtheorem{remark}[lem]{Remark}
\newtheorem{prop}[lem]{Proposition}
\newtheorem{thm}[lem]{Theorem}

\newtheorem {Def}[lem] {Definition}

\usepackage{color}
\definecolor{lilas}{RGB}{182, 102, 210}

\numberwithin{equation}{section}

\def\ben#1{\begin{equation}#1\end{equation}}

\title[Maximal edge-traversal time in First-passage percolation]
{Maximal edge-traversal time in First-passage percolation}
\date{\today}

\author{Shuta Nakajima} 
\address[Shuta Nakajima]
{University of Basel, Basel, Switzerland}
\email{shuta.nakajima@unibas.ch}


\begin{document}
\maketitle

\begin{abstract}
 In this paper, we study the maximal edge-traversal time on optimal paths in First-passage percolation on the lattice $\Z^d$ for several edge distributions, including the Pareto and Weibull distributions. It is known to be unbounded when the edge distribution has unbounded support [J. van den Berg and H. Kesten. Inequalities for the time constant in first-passage percolation. Ann. Appl. Probab.
56-80, 1993]. We determine the order of the growth depending on the tail of the edge distribution.
\end{abstract}
\maketitle

\section{Introduction}
First Passage Percolation (FPP) is a model of the spread of a fluid through a random medium which was first introduced by Hammersley and Welsh in 1965. In FPP, a graph with random weights is given and we consider the optimization problem of the passage time between two fixed vertices. The minimum value is called the first passage time and it represents the time when the fluid reaches from one point to the other. From the viewpoint of an optimization problem, properties of the optimal path that attains minimum value are also of interest. Theoretical physicists predicted that the front of spread in FPP asymptotically satisfies KPZ-equation \cite{KPZ86} in some sense. Moreover, they have found the relationship between the fluctuation of surface and the deviation of optimal paths, the so-called {\it scaling relation} \cite{KS91}. Over 50 years, as mathematical techniques have been developed for these problems, there has been significant progress especially about the asymptotic and the fluctuation of the first passage time and the surface growth. On the other hand, not much is known about the properties of the optimal path. The above-mentioned scaling relation concerns the geometry of the optimal path but it has not been proved fully rigorously \cite{Chat}. This paper studies the maximal weight of the edges on an optimal path aiming to provide a better understanding of how the medium along the optimal path looks. For more on the background and known results in FPP, we refer the reader to \cite{ADH}. 
\subsection{The setting of the model}
In this paper, we consider the first passage percolation on the lattice $\Z^d$. The model is defined as follows. An element of $\Z^d$ is called a vertex. Denote by ${\rm E}^d$ the set of non-oriented edges of the lattice $\Z^d$:
$${\rm E}^d=\{\langle v,w \rangle|~v,w\in\Z^d,~|v-w|_1=1\},$$
where $|u|_1=\sum^d_{i=1}|u_i|$ for $u\in\Z^d$. We say that $v$ and $w$ are adjacent if $|v-w|_1=1$. With a slight abuse of notation, an edge $e=\langle v,w \rangle$ is considered as a subset of $\Z^d$ such as $e=\{v,w\}$. We assign a non-negative random variable $\tau_e$ on each edge $e$. Assume that the collection $\tau=\{\tau_e\}_{e\in {\rm E}^d}$ is independent and identically distributed with a common distribution $F$. Let $(\Omega,\mathcal{F},\P)$ be the probability space and denote by $\E$ its expectation. A {\it path} $\gamma$ is a finite sequence $(x_0,\cdots,x_l)$ of $\Z^d$ such that for any $1\le i\le l$, $x_i$ and $x_{i-1}$ are adjacent. When $x_0=v$ and $x_l=w$ for a path $\gamma=(x_0,\cdots,x_l)$, we write $\gamma:v\to w$ and then $\gamma$ is said to be a path from $v$ to $w$. It is sometimes convenient to regard a path $\gamma=(x_0,\cdots,x_l)$ as a sequence of edges such as $(\langle x_0,x_1\rangle,\cdots\langle x_{l-1},x_l\rangle)$. Thus we will use this convention with some abuse of notation.

Given a finite path $\gamma$, we define the {\it passage time} of $\gamma$ as
$$\rmT(\gamma)=\sum_{e\in\gamma}\tau_e.$$
Given two vertices $v,w\in\Z^d$, we define the {\it first passage time} from $v$ to $w$ as
$$\rmT(v,w)=\inf_{\gamma:v\to w}\rmT(\gamma),$$
where the infimum is taken over all finite paths from $v$ to  $w$. We say that a finite path $\gamma:v\to w$ is {\it optimal} if $\rmT(\gamma)=\rmT(v,w)$. Denote by $\O(v,w)$ the set of all optimal paths from $v$ to $w$:
$$\O(v,w)=\{\gamma:v\to w|~\rmT(\gamma)=\rmT(v,w)\}.$$
It is proved in \cite{Kes86} that if $\P(\tau_e=0)\neq p_c(d)$, then at least one optimal path exists for any endpoints and if $\P(\tau_e=0)<p_c(d)$, then $\O(v,w)$ is always a finite set, where $p_c(d)$ is the critical probability of $d$-dimensional percolation. It should be noted that the same questions are still open when $\P(\tau_e=0)= p_c(d)$.

It is easy to see that $\rmT:\Z^d\times\Z^d\to\R_{\geq 0}$ is pseudometric. Hence optimal paths are sometimes called {\it geodesics}. Given a path $\gamma$, we set the {\it maximal edge-traversal time} (or the {\it maximal weight}) of $\gamma$ as 
$$\mathcal{M}(\gamma)=\sup_{e\in \gamma}\tau_e.$$
An edge $e$ is said to be a maximal edge for $\gamma$ if $e$ belongs to $\gamma$ and it attains the maximal weight of $\gamma$. In this paper, we investigate the growth rate of the maximum weight of optimal paths.
\subsection{Overview}
Edge weights along an optimal path are important research topics. In addition, there are many applications in the study of the geometry of optimal paths. For examples, Bates \cite{Bates20} proves that the empirical distribution of an optimal path, $\sum_{e\in \gamma} \delta_{\tau_e}$ for $\gamma\in\O(0, n\mathbf{e}_1)$ where $\delta_x$ is the Dirac measure at $x$, converges to some distribution independent of a choice of the optimal path, and as a corollary, he obtained the law of large numbers of the length of the optimal path. Also, the authors in \cite{JLS20} studied the tail bounds of the empirical distributions.

Van den Berg and Kesten proved in \cite{BK93} that the maximal edge-traversal times of optimal paths go to infinity as the endpoint goes to infinity if the distribution $F$ is unbounded, as a special case of more general theorems. The problem of the speed of the divergence is natural to consider and appears in \cite{ADH} as an open problem (Open problem 2).

Since the formulation of our results will be a bit complicated, we state some illustrative consequences first when the distribution is continuous, i.e. $\P(\tau_e=a)=0$ for any $a\in \R$. (See Thoerem~\ref{thm1} and Remark~\ref{rem1}.) If $\P(\tau_e\geq t)= \exp\{-t^{r+o(1)}\}$ with $r>0$ for any sufficiently large $t\in\R$, then the following occurs with high probability: for any optimal path $\pi$ from $0$ to $x$,
\begin{equation}\label{illustrate}
  \mathcal{M}(\pi)=\begin{cases}
(\log{|x|_1})^{\frac{1}{1+r}+o(1)}, &\text{ if }r\le d-1,\\
(\log{|x|_1})^{\frac{1}{d}+o(1)}, &\text{ if }d-1<r\leq d,\\
    (\log{|x|_1})^{\frac{1}{r}+o(1)}, &\text{ if }d<r,\\
  \end{cases}
  \end{equation}
where $o(1)\to 0$ as $|x|\to \infty$. It should be noted that if, in addition, we remove $o(1)$ in the assumption about the distributions, then we get the same results without $o(1)$ except when $r=d$ or $r=d-1$. And, if $ \alpha_1 t^{-\beta}\leq\P(\tau_e\geq t)\le  \alpha_2 t^{-\beta}$ for any sufficiently large $t>0$ with some constants $\beta>2$ and $\alpha_1,\alpha_2>0$, then the order becomes
$$\mathcal{M}(\pi)\asymp \frac{\log{|x|_1}}{\log{\log{|x|_1}}}.$$
These transitions itself are interesting and closely related to that of the large deviations of the first passage time \cite{CGM,CN21}.

In the proofs, one can see that different geometric pictures around the maximal edges appear. It would be interesting if there is a transition of some quantity about the configurations around the maximal edge, such as the average weight around the maximal edge.\\

\subsection{Main results}
We only consider the optimal paths from $0$ to $N\mathbf{e}_1$, though all of the results also hold for any direction. For the sake of the simplicity, we write ${\rm T}_N={\rm T}(0,N\mathbf{e}_1)$ and $\O_N=\O(0,N\mathbf{e}_1)$. Given $d\ge 2$ and $r\ge 0$, we set ${\rm f}_{d,r}(N)$ as
\begin{equation}\label{order}
  {\rm f}_{d,r}(N)=\begin{cases}
    (\log{N})(\log{\log{N}})^{-1}, &\text{ if }r=0,\\
  (\log{N})^{\frac{1}{1+r}}, &\text{ if }0<r< d-1,\\
  (\log{N})^{\frac{1}{d}}(\log{\log{N}})^{\frac{d-2}{d}}, &\text{ if }r=d-1,\\
  (\log{N})^{\frac{1}{d}} , &\text{ if }d-1< r< d,\\
      (\log{N})^{\frac{1}{d}}(\log{\log{N}})^{-\frac{1}{d}}, &\text{ if }r= d,\\
  (\log{N})^{\frac{1}{r}} , &\text{ if }d<r.\\
  \end{cases}
\end{equation}
 First we state the upper bound for maximal weight.
\begin{thm}\label{thm1}
Let $d\ge 2$. Suppose that there exist constants $r\in(0,\infty)$, $b,$ $a>0$ such that for $t\geq 0$, 
\begin{equation}\label{UB condition}
  \begin{split}
\P(\tau_e\geq{}t)\le a e^{-bt^{r}},
  \end{split}
\end{equation}
and $\P(\tau_e=0)<p_c(d).$ Then, there exists a positive constant $C$ such that 
$$\lim_{N\to\infty}\P\left(\sup_{\gamma\in\text{$\O_N$}}\mathcal{M}(\gamma)\leq{}C{}{\rm f}_{d,r}(N)\right)=1.$$
\end{thm}
\begin{thm}\label{thm2}
  Let $d\ge 2$. Suppose that $\E\tau_e^2<\infty$ and $\P(\tau_e=0)<p_c(d).$ Then, there exists a positive constant $C$ such that 
$$\lim_{N\to\infty}\P\left(\sup_{\gamma\in\text{$\O_N$}}\mathcal{M}(\gamma)\leq{}C{}{\rm f}_{d,0}(N)\right)=1.$$
\end{thm}
Next we move on to the lower bound. Let us restrict our attention to the following class of distributions, called useful  distributions, which is introduced in \cite{BK93}.
\begin{Def}\label{def:useful}
 Let $\underline{\tau}$ be the infimum of the support of a random variable $\tau_e$. We say that the distribution $F$ is useful if either holds:
$$(1)~\underline{\tau}=0\text{ and }\P(\tau_e=0)<p_c(d)\text{ or }(2)~\underline{\tau}>0\text{ and }\P(\tau_e=\underline{\tau})<\vec{p}_c(d),$$
where $p_c(d)$ and $\vec{p}_c(d)$ stand for the critical probabilities of $d$-dimensional percolation and oriented percolation model, respectively.
\end{Def}
 Note that if $F$ is continuous, then it is also useful. This restriction assures that a typical optimal path never goes far away taking only the minimum value of the distribution.
\begin{thm}\label{thm3}
  Let $d\ge 2$. Suppose that the following three hold:
  \begin{enumerate}
  \item $F$ is useful,
  \item for any positive integer $m$, $\E[\tau^m_e]<\infty$,
  \item there exist $r\in(0,\infty)$, $\beta,\alpha,\rho>0$ and $\kappa>1$ such that for any $t>\rho$, $$\P(t<\tau_e<\kappa{}t)\geq{}\alpha{}e^{-\beta{}t^r}.$$
  \end{enumerate}
  Then, there exists a positive constant $c$ such that
$$\lim_{N\to\infty}\P\left(\inf_{\gamma\in\text{$\O_N$}}\mathcal{M}(\gamma)\geq{}c{\rm f}_{d,r}(N)\right)=1.$$
\end{thm}
\begin{remark}\label{rem1}
  Let us comment how to deduce the lower bound in \eqref{illustrate} under the assumption $\P(\tau_e\geq t)=\exp{(-t^{r+o(1)})}$. We now consider a weaker assumption that $\kappa$ depends on $t$ such that $\kappa=t^{\delta}$ with $\delta>0$. Essentially the same proof yields the following result: for any $\e>0$, there exists $\delta>0$ such that under the above conditions (i)--(iii) with $\kappa=t^{\delta}$,
  $$\lim_{N\to\infty} \P\left( \inf_{\gamma\in\O_N}\mathcal{M}(\gamma)\geq c{\rm f}_{d,r}(N)^{1-\e}\right)=1.$$
  Note that if $\P(\tau_e\geq t)=\exp{(-t^{r+o(1)})},$ then it is easy to check
  that for any $\delta>0$, $$\P(t\leq \tau_e\leq t^{1+\delta})\geq e^{-t^{r+o(1)}},$$  and that the other conditions hold. Therefore, the lower bound of \eqref{illustrate} follows.
  \end{remark}
\begin{thm}\label{thm4}
  Let $d\ge 2$. Suppose that the following three hold:
  \begin{enumerate}
  \item  $F$ is useful,
    \item  $\E[\tau^{2}_e]<\infty$, 
    \item there exist $\beta,\alpha,\rho>0$ and $\kappa>1$ such that for any $t>\rho$, $$\P(t<\tau_e<\kappa{}t)\geq{} \alpha t^{-\beta}.$$
  \end{enumerate}
  Then, there exists a positive constant $c$ such that
$$\lim_{N\to\infty}\P\left(\inf_{\gamma\in\text{$\O_N$}}\mathcal{M}(\gamma)\geq{}c{\rm f}_{d,0}(N)\right)=1.$$
\end{thm}
\begin{remark}\label{rem:1}
 Given two sets $A,B$, we define ${\rm T}(A,B)=\inf_{x\in A,\,y\in B}{\rm T}(A,B)$ and denote by $\O(A,B)$ the set of corresponding optimal paths. If we consider the Box-to-Box first passage time $\rmT(D_{L_N}(0),D_{L_N}(N\mathbf{e}_1))$  instead of $\rmT(0,N\mathbf{e}_1)$ where $D_{L_N}(x)=x+[-L_N,L_N]^d$ for $x\in\R^d$ and $L_N$ to be specified below,
  and the maximal weight of corresponding optimal paths,  then the above four results hold not only in probability, but with probability one. More precisely, the following results hold:
\begin{prop}\label{prop1'}
 Let $L_N=\log{N}$. Under the condition of Theorem \ref{thm1}, the following happens with probability one: there exists a positive constant $C$ such that for any $N\in\N$, 
 \begin{equation}\label{upper-prop1}
   \sup_{\gamma\in\O(D_{L_N}(0),D_{L_N}(N\mathbf{e}_1))}\mathcal{M}(\gamma)\leq{}C{}{\rm f}_{d,r}(N).
   \end{equation}
If only we assume the condition of Theorem \ref{thm2}, then \eqref{upper-prop1} holds with $r=0$.
\end{prop}
\begin{prop}\label{prop3'}
  Take a positive constant $\rho$ and set $L_N=(\log{N})^{1+\rho}$. Under the condition of Theorem \ref{thm3}, the following happens with probability one: there exists a positive constant $c$ such that for any $N\in\N$ large enough,
 \begin{equation}\label{upper-prop2}
   \inf_{\gamma\in\O(D_{L_N}(0),D_{L_N}(N\mathbf{e}_1))}\mathcal{M}(\gamma)\geq{}c{}{\rm f}_{d,r}(N).
   \end{equation}
 If we assume the condition of Theorem \ref{thm4} instead,  then \eqref{upper-prop2} holds with $r=0$.
\end{prop}
 Remark~\ref{upp-rem} and~\ref{low-rem} explain the necessary modifications to show the above results.
\end{remark}
\begin{remark}
 In order to get the exact order of the growth of the maximal weight, in general, we need the assumption of the distribution such as $\P(\tau_e\geq t)\sim \exp\{-f(t)\}$, where $f(t)$ is regularly varying function, as we do for an extreme value problem of independent and identical distributed random variables. Therefore, our assumption on distributions is natural one.
  \end{remark}
\subsection{Notation and terminology}\label{section: notation}
This subsection introduces useful notations and terminologies  used in the proofs.
\begin{itemize}
    \item Given $a<b$, we write $\Iintv{a, b}=[a,b]\cap \Z.$
 \item Given a finite set $A$, we denote the cardinality of $A$ by  $\sharp A$.
  \item Given a path $\gamma=(x_0,\cdots,x_l)$, we define the length of $\gamma$ as  $\sharp \gamma=l$.
 \item Given two vertices $v,w\in\Z^d$ and a subset $D\subset\Z^d$, we set the {\it restricted} first passage time as
$$\rmT_D(v,w)=\inf_{\gamma\subset D}\rmT(\gamma),$$
 where the infimum is taken over all paths $\gamma\subset D$ from $v$ to $w$. If such a path does not exist, then we set the infinity instead. 
\item We use $c,$ $c_i$, $C$ and $C_i$ with $i\in \N$ for positive constants. They may change from line to line. Typically, $c$ and $c_i$ are used for  small constants and $C$ and $C_i$ for large ones.
\item The symbol $\lf\cdot\rf$ is a floor function, i.e.  $\lf x\rf$ is the greatest integer less than or equal to $x$.
\item It is useful to extend the definition to measure the $p$-norm between two sets as
  $${\rm d}_p(A,B)=\inf\{|x-y|_p|~x\in A,~y\in B\},\text{\hspace{4mm} $A,B\subset \R^d$}.$$
  When $A=\{x\}$, we simply write ${\rm d}_p(x,B)$.  We only use $p=1$ or $p=\infty$ in this article.
\item Given a set $D$ of $\Z^d$, let us define the inner boundary of $D$ as
  $$\partial{}D=\{v\in D|~\exists w\notin D\text{  s.t. }|v-w|_1=1\}.$$
  We define the interior of $D$ by $\iota(D)=D\backslash \partial D$.
  \item Given a set $D\subset \Z^d$ and $x,y\in D$, we write $x\sim_D y$ if there exists a path from $a$ to $b$ which lies only on $D$. Let us denote the connected component of $D$ containing of $x$ as ${\rm Conn}(x,D)=\{y\in D|~x\sim_D y\}.$
\end{itemize}
\subsection{Heuristics and Reader's guide}
\quad

For the proof of the upper bound, to each edge, we will consider a condition, where if the edge has a large weight and a path passes through the edge, then one can make a detour from this path to get a smaller passage time. We will check that all edges related to optimal paths satisfy this condition with high probability, and thus optimal paths do not have too large weights. This condition appears in Lemma~\ref{avoid on A}. The easiest case for the upper bound is $r= 1$ and proved in Section~\ref{section r=1 upper}.\\

 For the proof of the lower bound, we will use the resampling argument introduced in \cite{BK93}. We resample the local configurations to make the optimal paths pass through this region and take a sufficiently large weight. (See the conditions $\kE_{\eqref{kE1}}$--$\kE_{\eqref{kE3}}$ in Definitions~\ref{Def:A1}, \ref{Def:A2}, \ref{Def:A3} and Propositions \ref{crucial-prop}, \ref{crucial-prop2}, \ref{crucial-prop3}.) It needs the detailed information of optimal paths near the maximal edge, which heavily depends on the tail of distributions. The easiest case for the lower bound is $0<r<d-1$ and proved in Section \ref{mean to prob} and \ref{low:r<d-1}.

\section{Proof for the upper bound}
\subsection{General argument for upper bound}\label{r leq 1: upper}
Given an edge $e=\langle v,w\rangle$, we define $v_e\in \{v,w\}$ such that $|v_e|_1=\min\{|v|_1,|w|_1\}$ (such $v_e$ is uniquely determined) and  denote the $k$--th boundary and the set of its edges by ${\rm C}^{(e)}_{k}$ and $\tilde{\rmC}^{(e)}_{k}$, respectively (See Figure \ref{fig:one}):
\al{
{\rm C}^{(e)}_{k}&=\{z\in\Z^d:~|v_e-z|_{\infty}=k\},\\
\tilde{\rm C}^{(e)}_k&=\{\langle{}x,y\rangle:x,y\in{}{\rm C}^{(e)}_k~\text{ and }|x-y|_1=1\}.
}
Note that if $k\neq{}k'$, then  $\tilde{\rm C}^{(e)}_k\cap{}\tilde{\rm C}^{(e)}_{k'}=\emptyset$ and thus  $\{\tau_e\}_{e\in \tilde{\rm C}^{(e)}_k}$ and $\{\tau_e\}_{e\in \tilde{\rm C}^{(e)}_{k'}}$ are independent. Moreover, each face is the square of  sidelength $2k+1$ and its dimension is $d-1$. Thus there exists $C(d)>0$ which depends only on the dimension $d$ such that
\begin{equation}\label{number of C_k}
  \sharp \{v\in\Z^d|~v\in {\rm C}_k^{(e)}\}\leq C(d)k^{d-1}.
\end{equation}
In fact we can take $C(d)=4^dd$ since $\sharp \{v\in\Z^d|~v\in {\rm C}_k^{(e)}\}\leq 2d(2k+1)^{d-1}$.
\begin{figure}[b]
  \includegraphics[width=4.0cm]{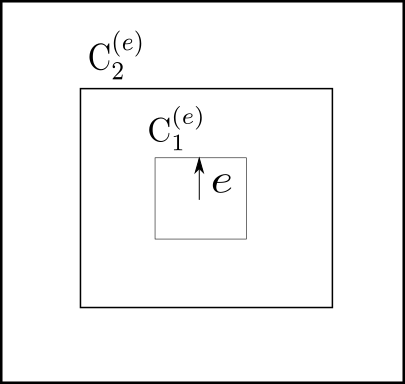}
\hspace{5mm}
  \includegraphics[width=6.0cm]{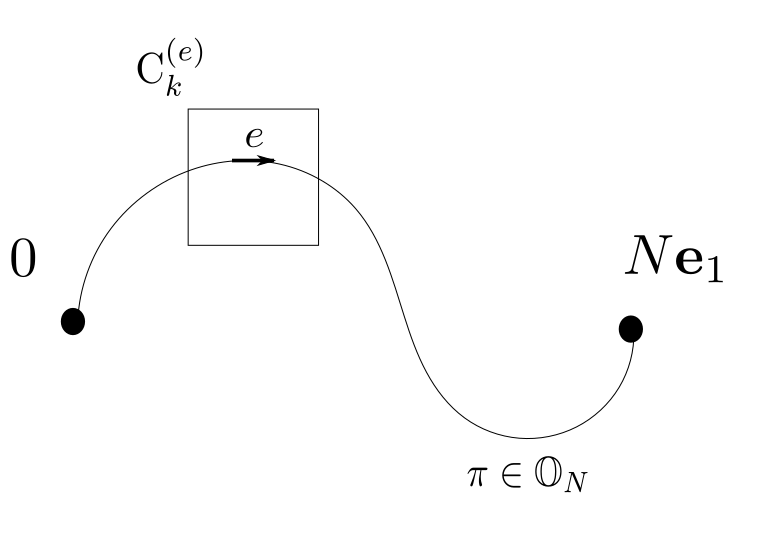}
\caption{}
Left: The figure of ${\rm C}^{(e)}_k$ and $\tilde{\rm C}^{(e)}_k$.\\
Right: We make a better path from the original one.
  \label{fig:one}
\end{figure}
\begin{Def}\label{good-upp}
  We say that $e$ is good if there exists $1\le k\leq{}{\rm f}_{d,r}(N)$ such that for any $v,w\in {\rm C}_k^{(e)}$, 
  \[ \rmT_{{\rm C}_k^{(e)}}(v,w)\leq{}M{\rm f}_{d,r}(N),\]
where $M$ will be chosen later.
\end{Def}
It will be proved in Lemma \ref{avoid on A} that for any path $\gamma$, if $e\in\gamma$ and $\tau_e>M{\rm f}_{d,r}(N)$, then the goodness of $e$ can make $\gamma$ detour with a smaller passage time. 
\begin{Def}\label{Def: B D}
  For $x\in\Z^d$ and $L>0$, we define
  \al{
        D_{L}(x)=(x+[-L,L]^d)\cap\Z^d,
          }
When $x=0$, we simply write $D_L=D_L(0)$.
\end{Def}
We take $K>0$ to be chosen in Lemma \ref{exit-prop}.
\begin{Def}
 We set
 \aln{
  \kE_{\eqref{kE1}}&=\{\forall{}\text{$e\in {\rm E}^d$ with $e\cap D_{KN}\neq{}\emptyset$, $e$ is good}\},\label{kE1}\\
  \kE_{\eqref{kE2}}&=\{\forall \gamma\in\O_N,~\gamma\subset{}D_{KN}\},\label{kE2}\\
  \kE_{\eqref{kE3}}&=\{\text{$\forall e\in {\rm E}^d$ with $e\cap{}(D_{M{\rm f}_{d,r}(N)}(0)\cup D_{M{\rm f}_{d,r}(N)}(N\mathbf{e}_1))\neq{}\emptyset$, }\tau_e\leq{}M{\rm f}_{d,r}(N)\}.\label{kE3}
  }
\end{Def}
We will see that the condition $\kE_{\eqref{kE1}}\cap \kE_{\eqref{kE2}}\cap \kE_{\eqref{kE3}}$ implies that the maximal weight of optimal paths is less than or equal to $M{\rm f}_{d,r(N)}$. 
\begin{lem}\label{avoid on A}
  On the event $\kE_{\eqref{kE1}}\cap \kE_{\eqref{kE2}}\cap \kE_{\eqref{kE3}}$, for any $\gamma\in\O_N$ and $e\in \gamma$,
  $$\tau_e\leq{}M{\rm f}_{d,r}(N).$$
\end{lem}
\begin{proof}
  Let us take $\gamma\in \O_N$ arbitrary and write $\gamma=\{x_0,\cdots,x_{\sharp \gamma}\}$. We fix $e=\langle x_t,x_{t+1} \rangle\in\gamma$ for some $0\le t\le \sharp \gamma-1$. If $e\cap D_{M{\rm f}_{d,r}(N)}(0)\cup D_{M{\rm f}_{d,r}(N)}(N\mathbf{e}_1)\neq \emptyset$, by $\kE_{\eqref{kE3}}$, then $\tau_e\leq M{\rm f}_{d,r}(N)$. Now we suppose that  $e\cap D_{M{\rm f}_{d,r}(N)}(0)\cup D_{M{\rm f}_{d,r}(N)}(N\mathbf{e}_1)= \emptyset.$ By $\kE_{\eqref{kE1}}$ and $\kE_{\eqref{kE2}}$, $e\subset D_{KN}$ and $e$ is good. Thus there exists $k\leq{}{\rm f}_{d,r}(N)$ such that for any $v,w\in {\rm C}_k^{(e)}$, 
  \begin{equation}\label{okkay}
   {\rm T}_{{\rm C}_k^{(e)}}(v,w)\leq{}M{\rm f}_{d,r}(N).
    \end{equation}
  We take such $k$. Let $x_p$ and $x_q$ be the first and last intersecting point between $\gamma$ and ${\rm C}_k^{(e)}$, i.e. $p=\inf\{s\in\{0,\cdots, \sharp \gamma\}|~x_s\in {\rm C}^{(e)}_k\}$ and $q=\sup\{s\in\{0,\cdots, \sharp \gamma\}|~x_s\in {\rm C}^{(e)}_k\}$. Since $e\cap{}(D_{M{\rm f}_{d,r}(N)}(0)\cup D_{M{\rm f}_{d,r}(N)}(N\mathbf{e}_1))={}\emptyset$, the inside of ${\rm C}_k^{(e)}$ contains neither $0$ nor $N\mathbf{e}_1$. Thus we have  $0 \le q<t<p\leq \sharp \gamma$. It follows from \eqref{okkay} that
  $$\tau_e\leq \rmT(x_p,x_q)\leq {\rm T}_{{\rm C}_k^{(e)}}(x_p,x_q)\leq{} M{\rm f}_{d,r}(N).$$ 
\end{proof}
From Lemma~\ref{avoid on A}, if we can prove 
\aln{
  \lim_{N\to\infty}\P(\kE_{\eqref{kE1}}\cap \kE_{\eqref{kE2}}\cap \kE_{\eqref{kE3}})=1,\label{probab converges to 1}
}
then the proof of Theorem~\ref{thm1} is completed.  First we will estimate $\P(\kE_{\eqref{kE2}})$.
\begin{lem}\label{exit-prop}
  Suppose $\E\tau_e^2<\infty$ and $\P(\tau_e=0)<p_c(d)$. Then there exist $C,K>0$ such that for any $N\in\N$, 
  \begin{equation}\label{estimate-A2}
    \P\left(\left(\bigcup_{\gamma\in\O_N}\gamma\right)\cap (D_{KN})^c \neq \emptyset\right)\leq CN^{-2d}.
    \end{equation}
  \end{lem}
\begin{proof}
  From Proposition 5.8 in \cite{Kes86}, there exist $C_1,C_2,C_3>0$ such that for any $\ell>0$,
   \begin{equation}\label{5.8}
     \begin{split}
       \P\left( \exists\text{ self avoiding path $\gamma$ starting at $0$
         with $\sharp \gamma\geq \ell$ and $\rmT(\gamma) < C_1 \ell$}\right) < C_2 \exp{(-C_3 \ell)}.
     \end{split}
     \end{equation}
  We take $K>4\E[\tau_e]/C_1$. Then,
  \begin{equation}\label{fin}
    \begin{split}
      &\P\left(\exists \gamma\in\O_N\text{  s.t. }\sharp \gamma\geq KN\right)\\
      &\leq \P\left(\exists \gamma\in\O_N\text{  s.t. }\sharp \gamma\geq KN\text{ and }\rmT_N< C_1KN\right)+\P\left(\rmT_N\geq C_1KN\right)\\
      &\leq C_2\exp{(-C_3KN)}+\P\left(\rmT_N\geq C_1KN\right),
  \end{split}
  \end{equation}
  where we have used \eqref{5.8} in the second inequality. Now we consider $2d$ disjoint paths $\{\gamma_i\}^{2d}_{i=1}$ from $0$ to $N\mathbf{e}_1$ so that $$\max\{\sharp \gamma_i|~i=1,\cdots,2d\}\leq 2N,$$
  as in \cite[p 135]{Kes86}. Since $\E[\rmT(\gamma_i)]\leq  C_1 K N/2$,
   \begin{equation}\label{fell}
     \begin{split}
       \P\left(\rmT_N\geq C_1KN\right)& \leq \prod^{2d}_{i=1} \P\left(\rmT(\gamma_i)\geq C_1KN\right)\\
       &\leq \prod^{2d}_{i=1} \P\left(|\rmT(\gamma_i)-\E[\rmT(\gamma_i)]|\geq C_1KN/2\right).
  \end{split}
   \end{equation}
By the Chebyshev inequality, this is further bounded from above with some constant $C_4=C_4(d,F,C_1)>0$ by
         \begin{equation}\label{fell2}
     \begin{split}
       &\quad \prod^{2d}_{i=1} \left((C_1KN/2)^{-2}\,2N \E[\tau_e^2]\right)\\
       &\leq C_4K^{-4d}N^{-2d}. 
  \end{split}
   \end{equation}
   Thus we have
    \begin{equation}\label{fell3}
      \begin{split}
        \P\left(\exists \gamma\in\O_N\text{  s.t. }\sharp \gamma\geq KN\right) &\leq  C_2\exp{(-C_3KN)}+C_4K^{-4d}N^{-2d}\\
   &\leq 2C_4K^{-4d}N^{-2d}.
  \end{split}
   \end{equation}
    Since
   $$\left(\bigcup_{\gamma\in\O_N}\gamma\right)\cap{}(D_{KN})^c\neq{}\emptyset \Longrightarrow   \max_{\gamma\in\O_N}\sharp\gamma\geq KN,$$
   we have
    \begin{equation}\label{fel}
      \begin{split}
        \P\left(\bigcup_{\gamma\in\O_N}\gamma\cap{}(D_{KN})^c\neq{}\emptyset\right)\leq 2C_4 K^{-4d} N^{-2d}. 
  \end{split}
    \end{equation}
  \end{proof}
Since the complement of $\kE_{\eqref{kE2}}$ is the event inside the probability in \eqref{estimate-A2},
we have
\begin{equation}\label{estimate-A2'}
  \P(\kE_{\eqref{kE2}}^c)\leq C N^{-2d},
  \end{equation}
which converges to $0$.
Next we will estimate $\P(\kE_{\eqref{kE3}})$. By the union bound, we have
    \begin{equation}\label{estimate-A3}
      \begin{split}
        \P(\kE_{\eqref{kE3}}^c)&=\P(\text{$\exists e\in {\rm E}^d$ s.t. $e\cap{}(D_{M{\rm f}_{d,r}(N)}(0)\cup D_{M{\rm f}_{d,r}(N)}(N\mathbf{e}_1))\neq{}\emptyset$ and }\tau_e>{}M{\rm f}_{d,r}(N))\\
        &\leq 2d\sharp{}(D_{M{\rm f}_{d,r}(N)}(0)\cup D_{M{\rm f}_{d,r}(N)}(N\mathbf{e}_1))\,a\exp{(-b(M{\rm f}_{d,r}(N)^r))}\\
        &\leq 2d( 4d M{\rm f}_{d,r}(N))^{d}\,a \exp{(-b M^r{\rm f}_{d,r}(N)^r)},
        \end{split}
         \end{equation}
    which also converges to $0$.   We will estimate $\P(\kE_{\eqref{kE1}})$ in several cases.

    \subsection{The case $r=1$}\label{section r=1 upper}
First, we  consider the case $r=1$. Then ${\rm f}_{d,r}(N)=\sqrt{\log{N}}$ and there exists $\beta>0$ such that $\E e^{\beta \tau_e}<\infty$. In this case, we take a positive constant $M$  such that

\begin{equation}\label{M-def1}
M>\beta^{-1} 16d^2 \E e^{\beta \tau_e}.
\end{equation}
Given two vertices $v,w\in {\rm C}_k^{(e)}$, we take a path $\gamma^w_v:v\to w$ lying on ${\rm C}_k^{(e)}$ whose length is at most $8d^2{\rm f}_{d,r}(N)$.  We will calculate the probability that $e$ is good.

Fix $v,w\in {\rm C}_k^{(e)}$. Then since we take $M$ sufficiently large as in \eqref{M-def1}, by the exponential Markov inequality, we have
\begin{equation}\label{2.1}
  \begin{split}
    \P\left(\sum_{\eta\in\gamma^w_v}\tau_{\eta}>{}M\sqrt{\log{N}}\right)&\leq{}\exp{(-\beta{}M\sqrt{\log{N}})}\prod_{\eta\in\gamma^w_v}\E{}e^{\beta\tau_{\eta}}\\
    &\leq{}\exp{(-\beta{}M\sqrt{\log{N}})}\left(\E{}e^{\beta\tau_e}\right)^{8d^2\sqrt{\log{N}}}\\
      &\leq{}\exp{\left(-\frac{\beta M\sqrt{\log{N}}}{2}\right)}.
  \end{split}
\end{equation}
   Recall that if $k\neq k'$, then $\{\tau_e\}_{e\in \tilde{\rm C}^{(e)}_k}$ and $\{\tau_e\}_{e\in \tilde{\rm C}^{(e)}_{k'}}$ are independent. It follows that  
\begin{equation}\label{2.4-}
  \begin{split}
    \P\left(e\text{ is not good}\right)&=\P\left( \forall k\leq \sqrt{\log{N}},~\exists v,w\in {\rm C}_k^{(e)}\text{ s.t. }\rmT_{{\rm C}_k^{(e)}}(v,w)>{}M{\rm f}_{d,r}(N)\right)\\
    &= \prod_{k\leq{}\sqrt{\log{N}}}\P\left(\exists v,w\in {\rm C}_k^{(e)}\text{ s.t. }\rmT_{{\rm C}_k^{(e)}}(v,w)>{}M{\rm f}_{d,r}(N)\right).
  \end{split}
\end{equation}
By $\rmT_{{\rm C}_k^{(e)}}(v,w)\leq \sum_{\eta\in\gamma^w_{v}}\tau_{\eta}$, for sufficiently large $N$, this is further bounded from above by 
\begin{equation}\label{2.4}
  \begin{split}
    & \prod_{k\leq{}\sqrt{\log{N}}}\P\left(\text{$\exists v,w\in {\rm C}^{(e)}_{k}$  s.t. $\sum_{\eta\in\gamma^w_{v}}\tau_{\eta}>{}M\sqrt{\log{N}}$}\right)\\
    \le& \prod_{k\leq{}\sqrt{\log{N}}}\left\{C(d)^2(\log{N})^{d-1}\max_{v,w\in {\rm C}_k^{(e)}} \P\left(\sum_{\eta\in\gamma^w_{v}}\tau_{\eta}>{}M\sqrt{\log{N}}\right)\right\}\\
  \le& \prod_{k\leq{}\sqrt{\log{N}}}\left\{C(d)^2(\log{N})^{d-1} \exp{\left(-\frac{\beta}{2}M\sqrt{\log{N}}\right)} \right\}\\
     \leq&{}\left(\exp{\left(-\frac{\beta}{4}M\sqrt{\log{N}}\right)}\right)^{\lf \sqrt{\log{N}}\rf}
    \le N^{-2d},
  \end{split}
\end{equation}
where we have used  \eqref{number of C_k}  and the union bound in the first inequality and \eqref{2.1} in the second inequality. By using the union bound again, we have
\al{
  \P(\kE_{\eqref{kE1}}^c)&\leq 2d(\sharp D_{KN})^2 \P(\text{$e$ is not good}).\\
  &\leq 2d (2KN+1)^d N^{-2d},
}
which converges to $0$ as $N\to\infty$.

Combined with \eqref{estimate-A2'} and \eqref{estimate-A3},  this yields
\begin{equation}
  \begin{split}
    &\P(\kE_{\eqref{kE1}}\cap \kE_{\eqref{kE2}}\cap \kE_{\eqref{kE3}})\\
    &\geq{}1-(\mathbb{P}(\kE_{\eqref{kE1}})+\mathbb{P}(\kE_{\eqref{kE2}})+\mathbb{P}(\kE_{\eqref{kE3}})),\label{prob of A}
  \end{split}
\end{equation}
and \eqref{probab converges to 1}.
 Combining with Lemma~\ref{avoid on A} completes the proof in the case $r=1$.
\subsection{The case $r\in (0,1)$}
Next we consider the case $0<r<1$, where ${\rm f}_{d,r}(N)=(\log{N})^{\frac{1}{1+r}}$. The proof is exactly the same as before except for the estimate of $\P(\kE_{\eqref{kE1}})$. In fact, by (4.2) of \cite{LDP}, \eqref{2.1} is replaced by
\begin{equation}\label{LDP-bound}
  \P\left(\sum_{\eta\in\gamma^w_v}\tau_{\eta}>M{\rm f}_{d,r}(N)\right)\leq{}e^{-cM^r{\rm f}_{d,r}(N)^r},
  \end{equation}
with some constant $c>0$ that depends only on the dimension $d$ and the distribution $F$. We have
\begin{equation}\label{2+}
  \begin{split}
    \P\left(e\text{ is not good}\right)  
 &\leq \prod_{k\leq{}{\rm f}_{d,r}(N)}\P\left(\text{$\exists v,w\in {\rm C}^{(e)}_{k}$  s.t. $\sum_{\eta\in\gamma^w_{v}}\tau_{\eta}>{}M{\rm f}_{d,r}(N)$}\right)\\
    \le& \prod_{k\leq{}{\rm f}_{d,r}(N)}\left\{C(d)^2(\log{N})^{d-1}\max_{v,w\in {\rm C}_k^{(e)}} \P\left(\sum_{\eta\in\gamma^w_{v}}\tau_{\eta}>{}M{\rm f}_{d,r}(N)\right)\right\}\\
    \leq&{}\left(\exp{\left(-\frac{c}{2}M^r {\rm f}_{d,r}(N)\right)}\right)^{\lf{\rm f}_{d,r}(N)\rf}
    \le N^{-2d},
  \end{split}
\end{equation}
  where we used the union bound in the second inequality. Using this and the union bound, we have $ \lim_{N\to\infty} \P(\kE_{\eqref{kE1}}^c)=0$ and \eqref{probab converges to 1} as desired.
\subsection{The case $1<r\le d$}
We consider the case $r\in (1,d]$, where
$${\rm f}_{d,r}(N)=\begin{cases}
  (\log{N})^{\frac{1}{d}}(\log{\log{N}})^{\frac{d-2}{d}}, &\text{ if }r=d-1,\\
  (\log{N})^{\frac{1}{d}} , &\text{ if }d-1< r< d,\\
      (\log{N})^{\frac{1}{d}}(\log{\log{N}})^{-\frac{1}{d}}, &\text{ if }r= d.\\
\end{cases}
$$
Note that in the previous arguments, the estimates of $\P(e\textrm{ is not good})$ are based on simple (sub-)exponential large deviations, see~\eqref{2.1}--\eqref{2.4} for example. It turns out that when $1<r\le d$ we need the following {\it super-exponential} tail estimates on the passage times. 
\begin{prop}[Lemma~4.5 in \cite{CN21}]\label{prop LDP}
  Let $d\geq{}2$. Suppose that the condition of Theorem \ref{thm1} holds with $r> 1$. For any $M_1>0$ there exists $M_2>0$ such that for any $e\in {\rm E}^d$, $L\ge 1$, $0\le k\le L$ and $v,w\in {\rm C}^{(e)}_{k}$,
  \begin{equation*}
    \begin{split}
     P\left( \rmT_{{\rm C}^{(e)}_{k}}(v,w)>{}M_2 L\right) &\le \exp{(-{\rm g}(r,d-1,L,k)M_1)},
     \end{split}
  \end{equation*}
  where
    \begin{equation*}
\begin{split}
{\rm g}(r,d,L,k)=\begin{cases}
    L^r&\text{ if $1<r<d$},\\
    \frac{L^d}{(1+\log{L})^{d-1}}&\text{ if $r=d$},\\
    L^d&\text{ if $r\in(d,d+1)$},\\
    L^{d+1}/k&\text{ if $r=d+1$}.
\end{cases}
\end{split}
\end{equation*}
\end{prop}
We take $M_1$ to be chosen later and set $M=M_2$ as in Proposition~\ref{prop LDP}. We use the same definitions as in Definition~\ref{good-upp}. 
Then for any sufficiently large $N\in\N$,
\begin{equation}\label{good:r>1}
  \begin{split}
    \P\left(\text{$e$ is not good}\right)&\le \prod_{1\le k\le{}{\rm f}_{d,r}(N)}\P\left(\text{$\exists v,w\in {\rm C}^{(e)}_{k}$\text{ s.t. } $\rmT_{{\rm C}^{(e)}_{k}}(v,w)>{}M{\rm f}_{d,r}(N)$}\right)\\
   &\leq{} (C(d)^2 {\rm f}_{d,r}(N)^{2d})^{{\rm f}_{d,r}(N)}\exp{\left(-\sum_{k=1}^{\lf{\rm f}_{d,r}(N)\rf} {\rm g}(r,d-1,{\rm f}_{d,r}(N),k)M_1\right)}.
  \end{split}
\end{equation}
If $1<r< d-1$, then since ${\rm f}_{d,r}(N)=(\log{N})^{\frac{1}{1+r}}$,
\begin{equation*}\begin{split}\sum_{k=1}^{\lf{\rm f}_{d,r}(N)\rf} {\rm g}(r,d-1,{\rm f}_{d,r}(N),k)&\geq \frac{1}{2}{\rm f}_{d,r}(N)^{1+r}\\
  &\geq \frac{1}{2}\log{N}.
  \end{split}\end{equation*}
If $r=d-1$, then since ${\rm f}_{d,r}(N)=(\log{N})^{\frac{1}{d}}(\log{\log{N}})^{\frac{d-2}{d}}\leq \log{N}$,
\begin{equation*}\begin{split}
    \sum_{k=1}^{\lf{\rm f}_{d,r}(N)\rf} {\rm g}(r,d-1,{\rm f}_{d,r}(N),k)&= \lf{\rm f}_{d,r}(N)\rf \frac{{\rm f}_{d,r}(N)^{d-1}}{(1+\log{{\rm f}_{d,r}(N))}^{d-2}}\\
  &\geq \frac{1}{2}{\rm f}_{d,r}(N)^d (\log{\log{N}})^{-(d-2)}\\
  &\geq \frac{1}{2}\log{N}.
\end{split}\end{equation*}
If $d-1<r<d$, then since ${\rm f}_{d,r}(N)=(\log{N})^{\frac{1}{d}}$,
\begin{equation*}
  \begin{split}\sum_{k=1}^{\lf{\rm f}_{d,r}(N)\rf} {\rm g}(r,d-1,{\rm f}_{d,r}(N),k)&\geq
  \frac{1}{2}{\rm f}_{d,r}(N)^{d}\\
  &\geq \frac{1}{2}\log{N}.
\end{split}\end{equation*}
If $r=d$, then since ${\rm f}_{d,r}(N)=(\log{N})^{\frac{1}{d}}(\log{\log{N}})^{-\frac{1}{d}}$,
\begin{equation*}\begin{split}\sum_{k=1}^{\lf{\rm f}_{d,r}(N)\rf} {\rm g}(r,d-1,{\rm f}_{d,r}(N),k)&\geq
  {\rm f}_{d,r}(N)^{d}\sum^{\lf{\rm f}_{d,r}(N)\rf}_{k=1}\frac{1}{k}\\
  &\geq \frac{1}{2d}\log{N},
\end{split}\end{equation*}
where we have used the following:
$$\sum^{\lf{\rm f}_{d,r}(N)\rf}_{k=1}\frac{1}{k}\geq{}\frac{1}{2d}\log{\log{N}}.$$
In all cases, by \eqref{good:r>1} and ${\rm f}_{d,r}(N)\log{{\rm f}_{d,r}(N)}\ll \log{N}$, if we take $M_1$ sufficiently large, then we have
$$\P(\text{ $e$ is not good })\le N^{-2d}.$$
As before, we can get $\lim_{N\to\infty} \P(\kE_{\eqref{kE1}}^c)=0$ and by Lemma~\ref{avoid on A},  \eqref{estimate-A2'} and \eqref{estimate-A3}, the proof is completed. 

\subsection{ The case $r>d$}
 We consider the case $r>d$, where ${\rm f}_{d,r}(N)=(\log{N})^{\frac{1}{r}}$. Recall the notation $D_L$ from Definition~\ref{Def: B D}.  For sufficiently large $M$,
\begin{equation*}\begin{split}
  \P\left(\text{ $\exists e\in {\rm E}^d$  s.t. $e\subset D_{KN}$ and $\tau_e>M{\rm f}_{d,r}(N)$}\right)&\le a\exp{(-b(M{\rm f}_{d,r}(N))^r)}\\
  &\leq N^{-2},
  \end{split}\end{equation*}
where $K$ is   chosen as  in Lemma \ref{exit-prop}. If for any $e\in {\rm E}^d$ with $e\subset D_{KN}$, $\tau_e\leq M{\rm f}_{d,r}(N)$ and $\kE_{\eqref{kE2}}$ holds, then
$$\max_{\gamma\in\O_N}\mathcal{M}(\gamma)\leq M{\rm f}_{d,r}(N) .$$
Thus, by \eqref{estimate-A2'}, the proof is completed.

\subsection{The case $r=0$}\label{UB r=0}
Let us move onto the case $\E\tau_e^2<\infty,$ where $  {\rm f}_{d,0}(N)=  \frac{\log{N}}{\log{\log{N}}}.$ In this case, since   we cannot expect any exponential bounds  and the estimate of $\P(e\text{ is not good})$ is not enough to get the desired bound, we slightly change the definition of goodness.
\begin{Def}
An edge $e$ is said to be $0$-good if there exists $1\leq k\leq{}M{\rm f}_{d,0}(N)$ such that for any $v,w\in {\rm C}_k^{(e)}$,
  $$\rmT_{{\rm C}_k^{(e)}}(v,w)\leq{}4M^2{\rm f}_{d,0}(N).$$
\end{Def}
Fix $k\leq M{\rm f}_{d,0}(N)$ and $v,w\in {\rm C}_k^{(e)}$. We consider $2(d-1)$ disjoint paths $\{\gamma_i\}^{2(d-1)}_{i=1}$ from $v$ to $w$  on ${\rm C}_k^{(e)}$ so that
$$\max\{\sharp \gamma_i|~i=1,\cdots,2(d-1)\}\leq 8d^2M{\rm f}_{d,0}(N),$$
as in \cite[p 135]{Kes86}.
If we take $M$ sufficiently large, then the  Chebyshev inequality yields that for any $i\in\{1\cdots,2(d-1)\}$,
\begin{equation*}
  \begin{split}
    \P\left( {\rm T}( \gamma_i)>2 M^2{\rm f}_{d,0}(N)\right)&\le  \P\left(\sum_{e\in \gamma_i}(\tau_{e}-\E\tau_e)>M^2{\rm f}_{d,0}(N)\right)\\
    &\le (M^2{\rm f}_{d,0}(N))^{-2}\E\left[\left(\sum_{e\in \gamma_i}(\tau_{e}-\E\tau_e)\right)^2\right]\\
    & \le (M^2{\rm f}_{d,0}(N))^{-2} (8d^2 M{\rm f}_{d,0}(N) \E \tau_e^2)\\
    &\le M^{-1}{\rm f}_{d,0}(N)^{-1}.
\end{split}
\end{equation*}
Thus we have
\begin{equation}\label{fell4}
  \begin{split}
  \P\left(\rmT_{{\rm C}^{(e)}_{k}}(v,w)>2M^2{\rm f}_{d,0}(N)\right)& \leq \prod_{i=1}^{2(d-1)}\P(\rmT(\gamma_i)>2M^2{\rm f}_{d,0}(N))\\
  &\leq M^{-2(d-1)}{\rm f}_{d,0}(N)^{-2(d-1)}.
  \end{split}
\end{equation}
Fix an arbitrary vertex $v_k\in {\rm C}_k^{(e)}$ to each $k\in\N$.  By the triangular inequality, for any $v,w\in {\rm C}_k^{(e)}$, $\rmT_{{\rm C}^{(e)}_k}(v,w)\leq \rmT_{{\rm C}^{(e)}_k}(v_k,v)+\rmT_{{\rm C}^{(e)}_k}(v_k,w)$. Thus, if there exist $v,w\in {\rm C}^{(e)}_{k}$ such that $\rmT_{{\rm C}^{(e)}_k}(v,w)>{}4 M^2 {\rm f}_{d,0}(N)$, then  there exists $z\in {\rm C}^{(e)}_{k}$ such that $\rmT_{{\rm C}^{(e)}_k}(v_k,z)>{}2M^2{\rm f}_{d,0}(N)$.   This yields
\begin{equation}\label{0-good-prob}
  \begin{split}
    \P\left(e\text{ is not $0$-good}\right)&=\P\left(\text{$\forall k\leq{}M{\rm f}_{d,0}(N),~\exists v,w\in {\rm C}^{(e)}_{k}$  s.t. $\rmT_{{\rm C}^{(e)}_{k}}(v,w)>{}4 M^2 {\rm f}_{d,0}(N)$}\right)\\
    &=\prod_{k\leq{}M{\rm f}_{d,0}(N)}\P\left(\exists v,w\in {\rm C}^{(e)}_{k}\text{  s.t. }\rmT_{{\rm C}^{(e)}_{k}}(v,w)>{}4 M^2 {\rm f}_{d,0}(N)\right)\\
    &\leq \prod_{k\leq{}M{\rm f}_{d,0}(N)}\P\left(\text{$\exists z\in {\rm C}^{(e)}_{k}$  s.t. $\rmT_{{\rm C}^{(e)}_k}(v_k,z)>{}2M^2 {\rm f}_{d,0}(N)$}\right).
  \end{split}
\end{equation}
Since  $\sharp{}{\rm C}_k^{(e)}\leq  C(d) k^{d-1}$ with some $C(d)>0$, as in \eqref{2.4}, if $M$ is sufficiently large, by \eqref{fell4}, then this is further bounded from above by
\begin{equation}\label{0-good-prob2}
  \begin{split}
    & \prod_{k\leq{}M{\rm f}_{d,0}(N)}C(d)\, k^{d-1}\max_{z\in {\rm C}_k^{(e)}} \P\left(\rmT_{{\rm C}^{(e)}_k}(v_k,z)>2 M^2 {\rm f}_{d,0}(N)\right)\\
    &\leq \left( C(d)(M{\rm f}_{d,0}(N))^{d-1}\right)^{M {\rm f}_{d,0}(N)} \left( M^{-2(d-1)}{\rm f}_{d,0}(N))^{-2(d-1)}\right)^{M {\lf \rm f}_{d,0}(N)\rf} \\
    &\le {\rm f}_{d,0}(N)^{-(d-1)M\lf {\rm f}_{d,0}(N)\rf}\leq{}N^{-2d}.
\end{split}
\end{equation}
Recall that $D_L(x)=(x+[-L,L]^d)\cap\Z^d.$ We define
\aln{
  \kE_{\eqref{kE1'}}&=\{\forall{}\text{$e\in {\rm E}^d$ with $e\cap D_{KN}\neq{}\emptyset$, $e$ is $0$-good}\},\label{kE1'}\\
  \kE_{\eqref{kE3'}}&=\{\forall v\in\partial{}D_{M{\rm f}_{d,0}(N)}(x),~\rmT(x,v)\leq{}2M^2{\rm f}_{d,0}(N),\text{ for $x=0,N\mathbf{e}_1$}\}.\label{kE3'}
}
By the union bound and \eqref{0-good-prob2}, we have $\lim_{N\to\infty}\P(\kE_{\eqref{kE1'}})=1$. Next we will prove  $\lim_{N\to\infty}\P(\kE_{\eqref{kE3'}})=1.$ For  $x=0,N\mathbf{e}_1$, we take $2d$ disjoint paths from $x$ to $v$ as in \eqref{0-good-prob} to obtain
$$\P\left(\exists v\in\partial{}D_{M{\rm f}_{d,0}(N)}(x)\text{ s.t. }\rmT(x,v)\geq{}2M^2{\rm f}_{d,0}(N)\right)\leq{}\sharp \partial D_{M{\rm f}_{d,0}(N)}(x) ({\rm f}_{d,0}(N))^{-2d}.$$ 
It follows that  $\lim_{N\to\infty}\P(\kE_{\eqref{kE3'}})= 1$. Thus
$$ \lim_{N\to\infty}\P(\kE_{\eqref{kE1'}}\cap \kE_{\eqref{kE2}}\cap \kE_{\eqref{kE3'}})= 1.$$
If there exist  $\gamma\in\O_N$ and an edge $e\in\text{$\gamma$}$ with $e\subset D_{M{\rm f}_{d,0}(N)}(0)$ such that $\tau_e\geq{}2M^2{\rm f}_{d,0}(N)$, then there exists $v\in\partial{}D_{M{\rm f}_{d,0}(N)}(0)$ such that $\rmT(0,v)\geq{}2M^2{\rm f}_{d,0}(N)$. Therefore, by the same proof as in  Proposition~\ref{avoid on A}, under $\kE_{\eqref{kE1'}}\cap \kE_{\eqref{kE2}}\cap \kE_{\eqref{kE3'}}$,
 $$\max_{\gamma\in\O_N}\mathcal{M}(\gamma)\leq 4 M^2 {\rm f}_{d,0}(N),$$
 and thus the proof is completed.

\begin{remark}\label{upp-rem}
  Let us comment how to prove Proposition \ref{prop1'}. When $r>0$, we replace $\kE_{\eqref{kE1}}\cap \kE_{\eqref{kE2}}\cap \kE_{\eqref{kE3}}$ by $\kE_{\eqref{kE1}}\cap \kE_{\eqref{kE2}}$. Indeed, Lemma \ref{avoid on A} can be proved by   exactly the same argument. Moreover, \eqref{prob of A} yields that $\P(\kE_{\eqref{kE1}}\cap \kE_{\eqref{kE2}})\geq{}1-CN^{-d}$ and the Borel-Cantelli lemma leads us to the conclusion.

  When $r=0$, we just replace $\kE_{\eqref{kE1'}}\cap \kE_{\eqref{kE2}}\cap \kE_{\eqref{kE3'}}$ by $\kE_{\eqref{kE1'}}\cap \kE_{\eqref{kE2}}$. The rest is the same as before.
  \end{remark}

\section{Proof for the lower bound}
 \subsection{From the means to the lower bounds}\label{mean to prob}
 Suppose that the condition of Theorem \ref{thm3} or Theorem \ref{thm4} holds. Let $c>0$ be a small positive constant. We take $\tilde{\tau}_e$ such that if $\tau_e< c{\rm f}_{d,r}(N)-1$, then $\tilde{\tau}_e=\tau_e$ and otherwise, $\tilde{\tau}_e=\tau_e+1$. We denote by $\tilde{{\rm T}}(x,y)$ the corresponding first passage time and write $\tilde{{\rm T}}_N=\tilde{{\rm T}}(0,N\mathbf{e}_1)$.  We denote by $\tilde{\O}_N$ the set of optimal paths for $\tilde{\rm T}_N$.  Obviously, $\tilde{{\rm T}}_N\geq {\rm T}_N$. Moreover, if $\min_{\O_N}\mathcal{M}(\gamma)<{}c{\rm f}_{d,r}(N)-1$, then $\tilde{{\rm T}}_N= {\rm T}_N$. By taking the contrapositive, we find that $\tilde{{\rm T}}_N> {\rm T}_N$ implies $\min_{\O_N}\mathcal{M}(\gamma)\geq c{\rm f}_{d,r}(N)-1$. We are going to prove that $\tilde{{\rm T}}_N> {\rm T}_N$ with high probability. The following statement will be proved in the next subsections.
 \begin{lem}\label{lem:mean}
   For any $\delta>0$, there exists $c>0$ such that for any $N\in\N$,
\begin{equation}
\begin{split}
  \E\left[ \min_{\gamma\in\widetilde\O_N}\sharp\{e\in{}\gamma|~\tilde{\tau}_e\geq{}c{\rm f}_{d,r}(N)\}\right]\geq{}cN^{1-\delta}. \label{mean}
  \end{split}
\end{equation}
 \end{lem}
 We will conclude the proof of Theorem~\ref{thm3} first by using Lemma~\ref{lem:mean}.  Note first that
$$\tilde{{\rm T}}_N\geq \rmT_N+\min_{\gamma\in\widetilde\O_N}\sharp\{e\in{}\gamma|~\tilde{\tau}_e\geq{}c{\rm f}_{d,r}(N)\}.$$
 In fact, if we take $\tilde{\gamma}\in\widetilde\O_N$, then
 \begin{equation*}\begin{split}
   \tilde{{\rm T}}_N&=\tilde{{\rm T}}(\tilde{\gamma})\\
   &= \rmT(\tilde{\gamma})+\sharp\{e\in{}\tilde{\gamma}|~\tilde{\tau}_e\geq{}c{\rm f}_{d,r}(N)\}\\
   &\geq \rmT_N+\min_{\gamma\in\widetilde\O_N}\sharp\{e\in{}\gamma|~\tilde{\tau}_e\geq{}c{\rm f}_{d,r}(N)\}.
   \end{split}\end{equation*}
 It follows from Lemma~\ref{mean} that 
\begin{equation}\label{diff-exp}
  \begin{split}
    \E{}\rmT_N+cN^{1-\delta}\leq \E{}\tilde{{\rm T}}_N.
  \end{split}
\end{equation}
Therefore, if both ${\rm T}_N$ and  $\tilde{{\rm T}}_N$ are well-concentrated around their means, then we can conclude  $\tilde{{\rm T}}_N>{\rm T}_N$ with high probability. For this purpose, we introduce the following concentration inequalities.
\begin{lem}\label{conc}
Suppose $\E\tau_e^2<\infty$. For any $\delta\in(0,1/4)$, there exists $C>0$ such that for sufficiently large $N$,
\begin{eqnarray}
  \P\left(|\rmT_N-\E{}\rmT_N|>N^{1-2\delta}\right)\leq{}CN^{-(1-4\delta)},\label{conc1}
\end{eqnarray}
\begin{eqnarray}
\P\left(|\tilde{{\rm T}}_N-\E{}\tilde{{\rm T}}_N|>N^{1-2\delta}\right)\leq{}CN^{-(1-4\delta) }.\label{conc22}
\end{eqnarray}
\end{lem}
\begin{proof}
The proof of this lemma follows from Theorem 3.1 in  \cite{ADH}, which was first proved in \cite{Kes93}. Indeed, since $\E\tau^{2}_e,\E\tilde{\tau}^{2}_e<C'$ with some constant $C'>0$ independent of $N$, Theorem 3.1 in  \cite{ADH} shows that 
\begin{eqnarray}
&\E[(\rmT_N-\E{}\rmT_N)^{2}]\le{}CN, \notag\\
&\E[(\tilde{{\rm T}}_N-\E{}\tilde{{\rm T}}_N)^{2}]\le{}CN,\label{moment2}
\end{eqnarray}
with some constant $C>0$. By the Chebyshev inequality, we have
\begin{equation*}
\begin{split}
\P\left(|\rmT_N-\E{}\rmT_N|>N^{1-2\delta}\right)&\leq{}N^{-2(1-2\delta)}\E[(\rmT_N-\E{}\rmT_N)^{2}]\\
&\leq{}CN^{-(1-4\delta)},
  \end{split}
\end{equation*}
which yields \eqref{conc1}. The same argument proves \eqref{conc22}.
\end{proof}
\begin{proof}[Proof of Theorem~\ref{thm3} and Theorem \ref{thm4} assuming Lemma~\ref{lem:mean}]
Let $\delta<1/4$. If both $|\rmT_N-\E{}\rmT_N|$ and $|\tilde{{\rm T}}_N-\E{}\tilde{{\rm T}}_N|$ are less than or equal to $N^{1-2\delta}$, then by Lemma~\ref{mean}, for sufficiently large $N\in\N$,
\begin{equation*}\begin{split}
  \tilde{{\rm T}}_N&\geq \E{}\tilde{{\rm T}}_N-N^{1-2\delta}\\
  &\geq \E{} \rmT_N+cN^{1-\delta}-N^{1-2\delta}\\
  &\geq \rmT_N+cN^{1-\delta}-2N^{1-2\delta}> \rmT_N.
\end{split}\end{equation*}
 Therefore, Lemma~\ref{conc} leads us to
\begin{equation*}
  \begin{split}
    &\P\left(\rmT_N=\tilde{{\rm T}}_N\right)\\
    &\le \P\left(|\rmT_N-\E{}\rmT_N|>N^{1-2\delta}\right)+\P\left(|\tilde{{\rm T}}_N-\E{}\tilde{{\rm T}}_N|>N^{1-2\delta}\right)\\
    &\le 2CN^{-(1-4\delta)}.
  \end{split}
\end{equation*}
Since $\min_{\O_N}\mathcal{M}(\gamma)<{}c{\rm f}_{d,r}(N)-1$ implies $\rmT_N=\tilde{{\rm T}}_N$, we have Theorem \ref{thm3} and Theorem \ref{thm4}.

\end{proof}
\begin{remark}\label{low-rem}
Let us comment how to prove  Proposition \ref{prop3'}. The proofs of \eqref{mean} can be  applicable  also in this case and Theorem 2 in \cite{Zhang2010} yields the better concentration bounds:
\begin{lem}\label{conc2}
 Suppose $\E[\tau_e^2]<\infty$. Then for any $m\in\N$ and $\delta<1/4$, there exists $N_0$ such that for any $N\ge N_0$,
\begin{eqnarray}
&~&\P\left(|\rmT(D_{L_N}(0),D_{L_N}(N\mathbf{e}_1))-\E{}\rmT( D_{L_N}(0),D_{L_N}(N\mathbf{e}_1))|>N^{1-2\delta}\right)\leq{}N^{-m},\label{conc1'}\\
&~&\P\left(|\tilde{{\rm T}}(D_{L_N}(0),D_{L_N}(N\mathbf{e}_1))-\E{}\tilde{{\rm T}}(D_{L_N}(0),D_{L_N}(N\mathbf{e}_1))|>N^{1-2\delta}\right)\leq{}N^{-m}. \label{conc2'}
\end{eqnarray}
\end{lem}
  Combining it with the previous arguments and the Borel-Cantelli lemma, we have the desired conclusion.
  \end{remark}
\subsection{Proof of Lemma~\ref{lem:mean}}
 Our goal is to prove \eqref{mean}. In this section, we explain the general arguments and we will evaluate what appear there in several cases in the following sections. The proof is based on the argument in \cite{BK93}, but the choice of box sizes and configurations inside of the box are considerably more complicated. The following lemma appears in Lemma 5.5 of \cite{BK93}.
\begin{lem}\label{lem2}
There exist $\delta_{\ref{lem2}}>0$ and $K>0$ such that for any $v,w\in\Z^d$,
$$\P\left(\rmT(v,w)<(\underline{\tau}+\delta_{\ref{lem2}})|v-w|_1\right)\leq{}e^{-K|v-w|_1}.$$
\end{lem}
 We fix $\delta_{\ref{lem2}}>0$ that satisfies Lemma \ref{lem2}. Note that Lemma \ref{lem2} also holds with $\tilde{\tau}$ since $\tilde{\tau}_e\geq \tau_e$. Remark that the usefulness of $F$ assumed in Theorem~\ref{thm3} and \ref{thm4} is used only in Lemma \ref{lem2} to prove \eqref{mean}. Since $\tilde{\tau}$ satisfies the condition of Theorem \ref{thm3} with the same $r$,  while the other parameters may  be different but independent of $N$, it suffices to show \eqref{mean} for $\tau$, i.e.
  \begin{equation}\label{new-mean}
\begin{split}
  \E\left[ \min_{\gamma\in \O_N}\sharp\{e\in{}\gamma|~\tau_e\geq{}c{\rm f}_{d,r}(N)\}\right]\geq{}cN^{1-\delta}.
  \end{split}
\end{equation}
  We fix constants $M,s_0,s_1>0$ such that $s_1\ll s_0 \ll M$ to be specified later. Set $n=\lf s_0 {\rm f}_{d,r}(N)\rf$ and $n_1=\lf s_1{\rm f}_{d,r}(N)\rf$, where $\lf \cdot\rf$ is a floor function.  We define three kinds of boxes whose notations are the same as in \cite{BK93} (see Figure \ref{fig:two}). First, we define hypercubes $S(l;n)$ for $l=(l_i)_{i=1}^d\in\Z^d$ as 
$$S(l;n)=\{v\in\Z^d:n l_i\le v_i< n(l_i+1),~1\le i\le d\}.$$
We call these hypercubes $n$-cubes. Second, we define large $n-$cubes $T(l;n)$ for $l\in\Z^d$ as $$T(l;n)=\{v\in\Z^d:n l_i-n\le v_i\le n(l_i+2),~1\le i\le d\}.$$
Finally, we define $n$-boxes ${\rm B}^j(l;n)$ for $l\in\Z^d$ and $j\in\{\pm 1,\cdots,\pm d\}$ as
$${\rm B}^j(l;n)=T(l;n)\cap{}T(l+2{\rm sgn}(j)\mathbf{e}_{|j|};n),$$
and its inner boundary $\partial{}{\rm B}^j(l;n)$ as
$$\partial{}{\rm B}^j(l;n)=\{v\in{}{\rm B}^j(l;n)|~\text{there exists }w\notin{}{\rm B}^j(l;n)\text{ s.t. }|v-w|_1=1\}.$$
Note that $S(l;n)\subset{}T(l;n)$ and ${\rm B}^j(l;n)$ is a closed box of size $3n\times\cdots\times{}3n\times{}n\times{}3n\cdots\times{}3n$, where $n$ is the length of $i$-th coordinate and $3n$ are the lengths of the other coordinates.
\begin{figure}[b]
\includegraphics[width=5.0cm]{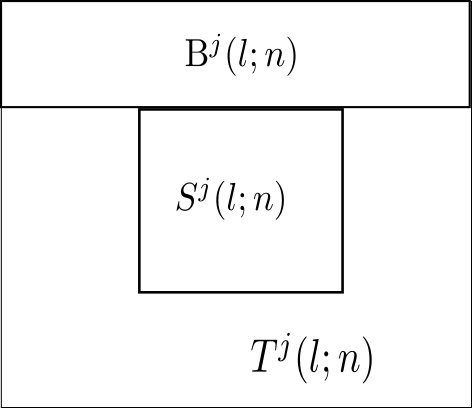}
\hspace{10mm}
\includegraphics[width=6.0cm]{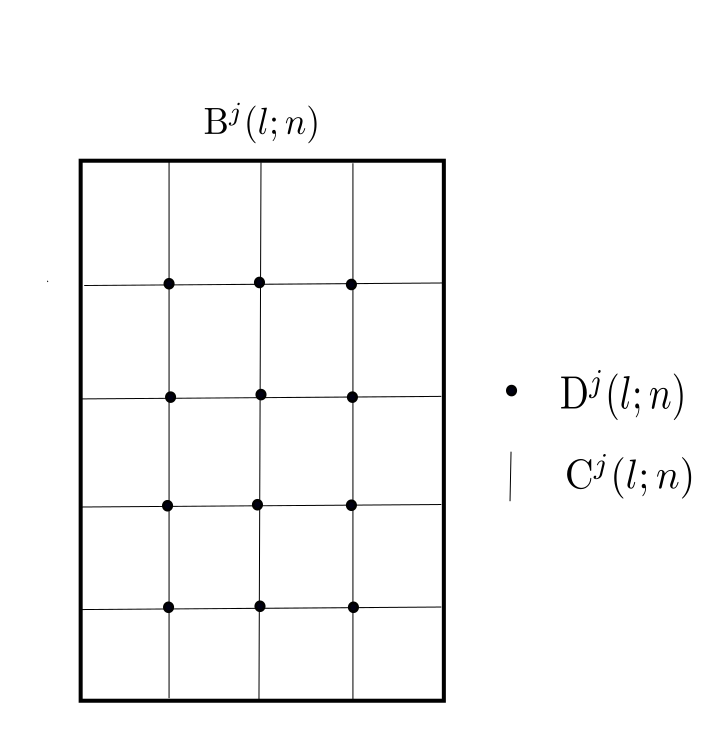}
\caption{}
Left: The figure of $S(l;n)$,  $T(l;n)$ and ${\rm B}^j(l;n)$.\\
Right: The figure of ${\rm C}(l;n)$ and ${\rm D}^j(l;n)$.
  \label{fig:two}
\end{figure}
  Let
\al{
{\rm D}^j(l;n)&=\{v\in{}{\rm B}^j(l;n)|~\text{$d_{\infty}(v,{\rm B}^j(l;n)^c)>{}n_1$, $v\in{}n_1\Z^d$}\},\\
{\rm C}^j(l;n)&=\left\{v+k \mathbf{e}_i|~v+k \mathbf{e}_i\in {\rm B}^j(l;n),~k\in\Z,~i\in\{1,\cdots,d\}, v\in{}{\rm D}^j(l;n)\right\}.\\
  \tilde{\rm C}^j(l;n)&=\{\langle v,w \rangle|~ v,w\in{}{\rm C}^j(l;n)\text{, }|v-w|_1=1\}.
}
  \begin{Def}
    We consider the following conditions:\\

    

(Black--1) For any $v,w\in {\rm B}^j(l;n)$ with $|v-w|_1\geq{}(\log{N})^{\frac{1}{4d(r+1)}}$,
    $$\rmT(v,w)\geq{}(\underline{\tau}+\delta_{\ref{lem2}})|v-w|_1,$$
     where $\delta_{\ref{lem2}}>0$ is the constant in Lemma \ref{lem2}.\\

(Black--2) For any $v,w\in\partial{}{\rm B}^j(l;n)$, 
$$\rmT_{\partial{}{\rm B}^j(l;n)}(v,w)\le {}M\left(|v-w|_1\lor (\log{N})^{\frac{1}{4d(r+1)}}\right).$$

(Black--3) For any edge $e\subset{} {\rm B}^j(l;n)$,
     $\tau_e\leq{}(\log{}N)^{\frac{1}{4d(r+1)}}.$\\

(Black--4) For any $v\in\partial{}{\rm B}^j(l;n)$, there exists $w\in \partial{}{\rm B}^j(l;n)\cap {\rm C}^j(l;n)$ such that, 
$$\rmT_{\partial{}{\rm B}^j(l;n)}(v,w)\le {}M n_1.$$

    When $r>0$, an $n$-box ${\rm B}^j(l,n)$ is said to be black if (Black--1), (Black--2) and  (Black--3) hold. When $r=0$, the $n$-box ${\rm B}^j(l,n)$ is called black if (Black--1) and (Black--4) hold. An $n$-cube $S(l,n)$ is said to be black if each of its surrounding $n$-boxes is black.
  \end{Def}
  The reason why $(\log{N})^{\frac{1}{4d(r+1)}}$ appears in the above definition will be clear in the following lemma, though the specific choice of the exponent is not that important. Note that $|{\rm C}^j(l;n)|\le C(s){\rm f}_{d,r}(N)$ with some positive constant $C(s)$ depending only on $s$, $s_1$ and $d$.
  \begin{lem}\label{black go to 1}
    If we take $M$ sufficiently large, then
    $$\P\left({\rm B}^j(l;n)\text{ is black }\right)\to{}1\text{ as }N\to{}\infty.$$
  \end{lem}
  \begin{proof}
    By Lemma \ref{lem2} and the union bound,
    \begin{equation*}
      \begin{split}
        &\P({\rm B}^j(l;n)\text{ does not satisfy (Black--1)})\\
        &\leq 2d\sharp {\rm B}^j(l;n)^2\max\{\P(\rmT(v,w)< (\underline{\tau}+\delta)|v-w|_1)|~v,w\in \Z^d,~|v-w|_1\geq (\log{N})^{\frac{1}{4d(r+1)}}\}\\
        &\leq 2d (3n)^{2d} \exp{(-K (\log{N})^{\frac{1}{4d(r+1)}})},
      \end{split}
    \end{equation*}
    which converges to $0$ and thus (Black--1) holds with high probability.
    
     Next, we consider (Black--2) with $r>0$. Let $v,w\in \partial {\rm B}^j(l;n)$ and $\gamma^v_w:v\to w$ be a path on $\partial {\rm B}^j(l;n)$ whose length is at most $6d|v-w|_1$.
  Since  $\E\tau_e^{2m}<\infty$ with $m=\lf 32d^2(r+1)\rf$, by the same argument as in Lemma~\ref{exit-prop}, for any $v,w\in \partial {\rm B}^j(l;n)$, we have
  \begin{align*}
    &\quad \P\left( \rmT_{\partial {\rm B}^j(l;n)}(v,w)>M(|v-w|_1 \lor(\log{}N)^{\frac{1}{4d(r+1)}}) \right)\\
    &\le \P\left( \rmT(\gamma^v_w)-\E \rmT(\gamma^v_w)>\frac{M}{2}(|v-w|_1 \lor(\log{}N)^{\frac{1}{4d(r+1)}}) \right)\\
    &\leq \left(\frac{M}{2}(|v-w|_1 \lor(\log{}N)^{\frac{1}{4d(r+1)}}) \right)^{-2m} \E[(\rmT(\gamma^v_w)-\E \rmT(\gamma^v_w))^{2m}]\\
    & \le (\log{N})^{-2d}.
  \end{align*}
  By $\sharp \partial {\rm B}^j(l;n)\le C(d){\rm f}_{d,r}(N)^{d-1}\le C(d)(\log{N})^{d-1}$ with some constant $C(d)>0$, (Black--2) holds with high probability  for $r>0$.

  By the union bound and $\P(\tau_e>(\log{}N)^{\frac{1}{4d(r+1)}})\leq C(\log{N})^{-2d}$, (Black--3) holds with high probability for $r>0$.

  Finally, we consider (Black--4) with $r=0$. We fix $v\in \partial{}{\rm B}^j(l;n)$. There exists $w\in \partial{}{\rm B}^j(l;n)\cap {\rm C}^j(l;n)$ such that $|v-w|_1\leq 2d n_1$. Consider $2(d-1)$ disjoing paths $(r_i)_{i=1}^{2(d-1)}$ from $v$ to $w$ on $\partial{}{\rm B}^j(l;n)$ so that $\sharp r_i\leq  4d n_1$ as in Section~\ref{UB r=0}. For $M$ large enough, using the Chebyshev inequality, we have
  \al{
    \P({\rm T}_{ \partial{}{\rm B}^j(l;n)}(v,w)>M n_1)&\leq \prod_{i=1}^{2(d-1)}\P({\rm T}_{ \partial{}{\rm B}^j(l;n)}(r_i)>M n_1)\\
    &\leq n_1^{2(d-1)}.
  }
  By the union bound, we conclude that
  \al{
    &\qquad \P(\exists v\in \partial{}{\rm B}^j(l;n)\,\text{s.t.}\,\forall w\in \partial{}{\rm B}^j(l;n)\cap {\rm C}^j(l;n),\, {\rm T}_{ \partial{}{\rm B}^j(l;n)}(v,w)>M n_1)\\
    &\leq ( \sharp \partial{}{\rm B}^j(l;n)) n_1^{2(d-1)}\leq C(d) n^{d-1} n_1^{2(d-1)},
  }
  which converges to $0$. Therefore,  (Black--4) holds with high probability for $r=0$.
  \end{proof}
   Combining the previous lemma and a similar argument (Peierls argument) of (5.2) in \cite{BK93} shows the following lemma. We skip the details.
\begin{lem}\label{lem:p}
There exist $\epsilon,u>0$ such that for any $N\in\N$,
 $$\P\left(\exists\text{$\gamma\in \O_N$ visiting at most $\frac{\epsilon{}N}{{\rm f}_{d,r}(N)}$ distinct black $n$-cubes}\right)\leq{}\exp{\left(-u\frac{N}{{\rm f}_{d,r}(N)}\right)}.$$
\end{lem}
We note that $\epsilon$ and $u$ above depend on $s_0,s_1,M$ but not on $N$. A path which starts in $S(l;n)$ and ends outside $T(l;n)$ must have a segment which lies entirely in one of the surrounding $n$-boxes, and which connects the two opposite large faces of that $n$-box. This means that the path crosses the $n$-box in the short direction (See Figure \ref{fig:three}). Hereafter  ``crossing an $n$-box'' means crossing in the short direction. From this and Lemma \ref{lem:p}, we have
\begin{equation}\label{estimate-Black-box}
  \E[\sharp\{\text{distinct black $n$-box ${\rm B}^j(l;n)$  s.t. $\exists \gamma \in \O_N$ crossing ${\rm B}^j(l;n)$} \}]\geq{}\frac{\epsilon{}N}{2{\rm f}_{d,r}(N)}.
  \end{equation}
Fix some small constant $c>0$ depending on $s$, to be chosen later.
\begin{Def}
  An $n$-box ${\rm B}^j(l;n)$ is said to be good if for any $\gamma\in \O_N$, there exists $e\in\gamma$ such that both vertices of $e$ are in ${\rm B}^j(l;n)$ and $\tau_e\geq{}c{\rm f}_{d,r}(N)$.
\end{Def}
Note that
\begin{equation}\label{OT1}
\begin{split}
&\E\left[\inf_{\gamma\in \O_N}\sharp\{e\in{}\gamma:~ \tau_e\geq{}c{\rm f}_{d,r}(N)\}\right]\\
  &\geq{}\frac{1}{2d}\E[\sharp\{(j,l)|\text{${\rm B}^j(l;n)$ is a good $n$-box} \}]\\
  &= \frac{1}{2d}\sum_{(j,l)}\P({\rm B}^j(l;n)\text{ is a good $n$-box}),
  \end{split}
\end{equation}
where $2d$ appears because of the overlap of $n$-boxes. On the other hand, \eqref{estimate-Black-box} yields
\begin{equation}\label{OT2}
  \begin{split}
   \frac{1}{2}\frac{\epsilon{}N}{{\rm f}_{d,r}(N)} &\leq {}\E[\sharp\{\text{distinct black $n$-box ${\rm B}^j(l;n)$  s.t. $\exists \gamma\in$$\O_N$ crosses ${\rm B}^j(l;n)$} \}]\\
  &= \sum_{(j,l)}\P(\text{${\rm B}^j(l;n)$ is black and $\exists \gamma\in \O_N$ crosses ${\rm B}^j(l;n)$}).
  \end{split}
\end{equation}
The next proposition will be proved in subsequent sections, which implies that black boxes can be made good boxes  without too much cost.
\begin{prop}\label{crucial-don}
  With some choices of $c,s_0,s_1,M$, there exists $\delta>0$ such that for $N$ large enough,
 \begin{equation*}
\begin{split}
  &\qquad \P({\rm B}^j(l;n)\text{ is good})\\
  &\geq N^{-\delta}\P\left(\text{${\rm B}^j(l;n)$ is black and $\exists\gamma\in\O_N$ crosses ${\rm B}^j(l;n)$}\right).
  \end{split}
\end{equation*}
\end{prop}
\begin{proof}[Proof of Lemma~\ref{lem:mean} assuming Proposition~\ref{crucial-don}.]
 Combining Proposition~\ref{crucial-don} with \eqref{OT1} and \eqref{OT2}, we have
\begin{equation*}
\begin{split}
&\E\left[\inf_{\gamma\in \O_N}\sharp\{e\in{}\gamma:~\tau_e\geq{}c{\rm f}_{d,r}(N)\}\right]\\
&\geq{}\frac{1}{2d}N^{-\delta}\E[\sharp\{\text{distinct black $n$-box ${\rm B}^j(l;n)$  s.t. $\exists \gamma\in$$\O_N$ crosses ${\rm B}^j(l;n)$} \}]\\
&\geq{}\frac{1}{4d}\frac{\epsilon{}N^{1-\delta}}{{\rm f}_{d,r}(N)},
  \end{split}
\end{equation*}
which proves Lemma~\ref{new-mean},  retaking $\delta>0$.
\end{proof}
\subsection{Lower bound for $0< r<d-1$ or $r=1$}\label{low:r<d-1}
Set $n=\lf {\rm f}_{d,r}(N)\rf$ and $n_1=\lf s{\rm f}_{d,r}(N)\rf$, where $s$ is chosen later.  Our goal is to prove Propositioin~\ref{crucial-don}. For the main step, we set
\begin{equation}\label{oppappi2}
\begin{split}
&\tilde{\rm E}^j(l;n)=\{\langle{}v,w\rangle|~v\in{}{\rm C}^j(l;n)\backslash\partial {\rm B}^j(l;n)\text{ and }w\notin{}{\rm C}^j(l;n)\backslash\partial {\rm B}^j(l;n)\}.
\end{split}
\end{equation}
Let us note that any path has to pass through at least one of edges of $\tilde{\rm E}^j(l;n)$ to enter  ${\rm C}^j(l;n)\backslash \partial {\rm B}^j(l;n)$.
It is easy to see that there exists $C(s)>0$ such that
\begin{equation}\label{oppappi}
  |\tilde{\rm C}^j(l;n)|,~|\tilde{\rm E}^j(l;n)|\le C(s){\rm f}_{d,r}(N).
\end{equation}
\begin{Def}\label{Def:A1}
   We say that the collection $\tau=\{\tau_e\}_{e\in {\rm E}^d}$ satisfies $A_1$--condition if 
   $$ \begin{cases}
  c{\rm f}_{d,r}(N)\leq{}\tau_e\leq{}\kappa{}c{\rm f}_{d,r}(N) & \text{when $e\in{}\tilde{\rm E}^j(l;n)$}, \\
  \tau_e\leq{}\underline{\tau}+c&\text{when $e\in\tilde{\rm C}^j(l;n)\backslash\tilde{\rm E}^j(l;n)$},\\
   \end{cases}$$
  where $\kappa$ is  the parameter in Theorem \ref{thm3}. If $\tau$ satisfies $A_1$--condition, then we write $\tau\in A_1$.
\end{Def}
Under the $A_1$--condition,  the weights are atypically small on the grid (i.e. ${\rm C}^j(l;n)$) but a path has to pass through a large weight to enter the grid.

A resampled configuration $\tau^*=\{\tau^*_e\}_{e \in {\rm E}^d}$ is taken to be $\tau^*_e=\tau_e$ if $e\notin\tilde{\rm C}^j(l;n)\cup \tilde{\rm E}^j(l;n)$  and an independent copy of $\tau_e$ if $e\in \tilde{\rm C}^j(l;n)\cup \tilde{\rm E}^j(l;n)$. We enlarge the probability space so that it can measure the event both for $\tau$ and $\tau^*$ and we still denote the probability measure by $\P$.  We denote by ${\rm T}^{*}(\gamma)$ the passage time of a path $\gamma$ for $\tau^*$ and  by ${\rm T}^*(x,y)$ the corresponding first passage time and we write ${\rm T}^*_N={\rm T}^*(0,N\mathbf{e}_1)$. We also denote by $\O^*_N$ the set of all optimal paths for ${\rm T}^*_N$.  Let $A$ be the event defined as
\begin{equation}\label{A1'}
\begin{split}
  A=\{\tau^*\in A_1\}\cap \{
  {\rm B}^j(l;n)\text{ is black for $\tau$ and }\exists\gamma\in \O_N\text{ crosses }{\rm B}^j(l;n)\}.
\end{split}
\end{equation}
\begin{figure}[b]
  \includegraphics[width=6.0cm]{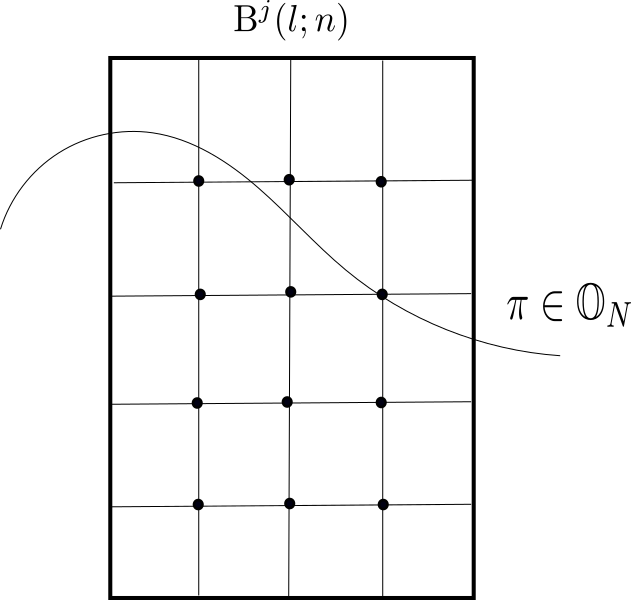}
  \hspace{5mm}
  \includegraphics[width=6.5cm]{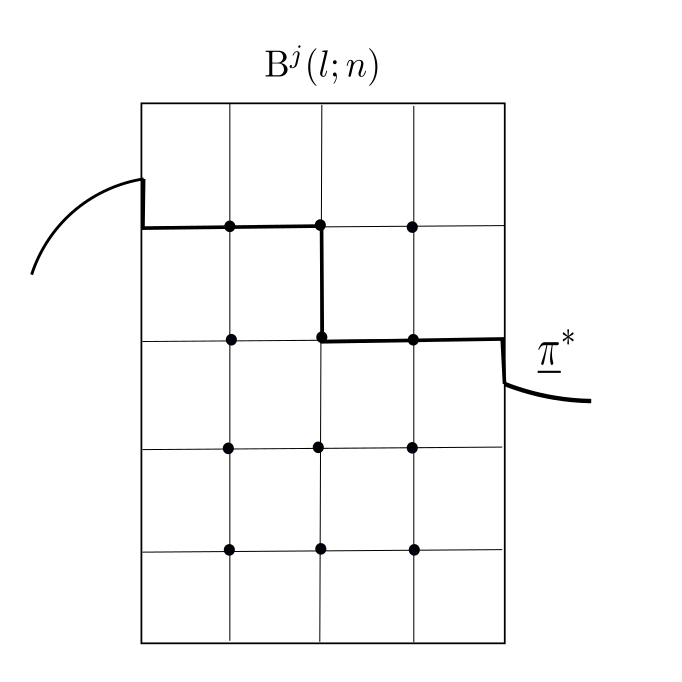}
  \caption{}
Left: $\O_N$ crosses a $n$-box in the short direction.\\
Right: How to construct a new path from $\O_N$.
 \label{fig:three}
\end{figure}
\begin{prop}\label{crucial-prop}
If we take $s$ sufficiently small depending on $M$ and $c$ sufficiently small depending on $s$, then for $N$ large enough, except when $0\in {\rm B}^j(l;n)$ or $N\mathbf{e}_1\in {\rm B}^j(l;n)$, we have
\begin{equation}\label{crucial}
\begin{split}
  &\P\left(\text{$n$-box ${\rm B}^j(l;n)$ is good for $\tau$}\right)\\
  &=\P\left(\text{$n$-box ${\rm B}^j(l;n)$ is good for $\tau^*$}\right)\\
  &\geq\P(A).
\end{split}
\end{equation}

\end{prop}
\begin{proof}
  Since $\tau$ and $\tau^*$ have the same distributions, the equality is trivial. Next, we consider the inequality above. It suffices to show $A\subset \text{\{$n$-box ${\rm B}^j(l;n)$ is good for $\tau^*$}\}$. Let us briefly explain the heuristics behind the proof. Under the event $A$, since a path can pass through the box ${\rm B}^j(l;n)$ with  a smaller passage time than that before resampling, any optimal path must enter the inside of ${\rm B}^j(l;n)$ after resampling. 
  Since we resample the configurations only of $\tilde{\rm E}^j(l;n)$ and $\tilde{C}^j (l;n)$,
   any optimal path must pass on an edge of $\tilde{\rm E}^j (l;n) \cup \tilde{C}^j (l;n)$, in particular an edge of $\tilde{\rm E}^j(l;n)$  as explained below \eqref{oppappi2}.
   By the $A_1$-condition, this implies ${\rm B}^j(l;n)$ is good. Let us make this rigorous.

   Assume that the event $A$ occurs.  To prove \eqref{crucial}, it suffices to show that
  \begin{equation}\label{diff}
    {\rm T}^*_N<\rmT_N.
    \end{equation}
  In fact, since we change configurations only on $\tilde{\rm C}^j(l;n)$ and $\tilde{\rm E}^j(l;n)$, \eqref{diff} yields that any $\pi^*\in \O_N^*$ has to pass through an edge of $\tilde{\rm C}^j(l;n)\cup \tilde{\rm E}^j(l;n)$ at least one time, since otherwise
  $${\rm T}^*_N={\rm T}^*(\pi^*)= \rmT(\pi^*)\geq \rmT_N.$$
  Moreover, in order to enter ${\rm C}^j(l;n)\backslash \partial {\rm B}^j(l;n)$, $\pi^*$ has to pass through an edge of $\tilde{\rm E}^j(l;n)$ at least one time. Therefore ${\rm B}^j(l;n)$ is good for $\tau^*$. This yields \eqref{crucial}.

 Let us prove \eqref{diff}. We take an arbitrary optimal path $\pi\in\O_N$. We construct a new path $\underline{\pi}^{*}$ from $\pi$ to prove \eqref{diff} as follows (see also  Figure \ref{fig:three}). Let $v$ and $w$ be the first intersecting point and the last intersecting point between $\pi$ and ${\rm B}^j(l;n)$, respectively.  Note that $\tau$ and $\tau^*$ are the same on $\partial{}{\rm B}^j(l;n)$. Under the assumption that ${\rm B}^j(l;n)$ is black, we can take $v_1,w_1\in{}\partial {\rm B}^j(l;n)\cap {\rm C}^j(l;n)$ and paths $\underline{\pi}^*_1:v\to{}v_1$ and $\underline{\pi}^*_2:w_1\to{}w{}$ on $\partial{}{\rm B}^j(l;n)$ such that $\max\{\rmT(\underline{\pi}^*_1),\rmT(\underline{\pi}^*_2)\}\le 2d Mn_1$ and $\max\{|v-v_1|_1,|w-w_1|_1\}\leq 2dn_1$. We take a path $\underline{\pi}^*_3\subset{}\tilde{\rm C}^j(l;n)\cup{}\tilde{\rm E}^j(l;n)$ from $v_1$ to $w_1$ such that $\underline{\pi}^*_3$ has exactly two edges in $\tilde{\rm E}^j(l;n)$ and at most $|v_1-w_1|_1+4n_1$ edges in $\tilde{\rm C}^j(l;n)$. For $x,y\in\pi$, we write $\pi|_{x\to y}$ the sub-path of $\pi$ from $x$ to $y$. Finally, we connect $\pi|_{0\to v}$, $\underline{\pi}^*_1$, $\underline{\pi}^*_3$, $\underline{\pi}^*_2$, $\pi|_{w\to N\mathbf{e}_1}$ in this order and let $\underline{\pi}^*$ be the new path. Note that since $\pi$ crosses ${\rm B}^j(l;n)$, $\max\{|x-y|_1|~x,y\in \pi\cap {\rm B}^j(l;n)\}\geq n$. Therefore, by (Black--1),
  \begin{equation}\label{new-passage3}
    \begin{split}
      \rmT(v,w)&\geq{}(\underline{\tau}+\delta_{\ref{lem2}})(|v-w|_1\lor n)\\
      &\geq \underline{\tau}|v-w|_1+\delta_{\ref{lem2}} n/2,
   \end{split}
\end{equation}
  and by $|v_1-w_1|_1\leq 3dn$ and $\sharp \underline{\pi}^*_3\leq |v_1-w_1|+4n_1\leq |v-w|_1+8n_1$, if we take $c$ sufficiently small depending on $s$, then we obtain
\begin{equation}\label{new-passage2}
\begin{split}
  {\rm T}^*(v,w)&\leq {\rm T}^*(\underline{\pi}^*_1)+{\rm T}^*(\underline{\pi}^*_2)+{\rm T}^*(\underline{\pi}^*_3)\\
  &\leq 2d Mn_1+2\kappa{}c{\rm f}_{d,r}(N)+(\underline{\tau}+c)(|v-w|_1+8n_1)+2d Mn_1\\
  &\leq \underline{\tau}|v-w|_1+8d Mn_1.
  \end{split}
\end{equation}

 Since we only resample  the edges in ${\rm B}^j(l;n)$ and the paths $\pi|_{0\to v}$ and $\pi|_{0\to w}$ does not use such edges,  we have
     \begin{equation*}
       \begin{split}
         &{\rm T}_N={\rm T}(0,v)+{\rm T}(v,w)+{\rm T}(w,N\mathbf{e}_1),\\
         &{\rm T}^*(0,v)\leq \rmT(0,v),\\
         &{\rm T}^*(w,N\mathbf{e}_1)\leq \rmT(w,N\mathbf{e}_1).
          \end{split}
     \end{equation*}
     Moreover,  by  the triangular inequality, we get
     $$\rmT^*_N \leq \rmT^*(0,v)+\rmT^*(v,w)+\rmT^*(w,N\mathbf{e}_1).$$
         Thus, using \eqref{new-passage3} and \eqref{new-passage2}, if $s$ is sufficiently small such that $100 d M  n_1\leq \delta_{\ref{lem2}} n$, then this proves
     \begin{equation}\label{new-passage}
       \begin{split}
        \rmT_N -{\rm T}^*_N         &\geq \rmT(0,v)+\rmT(v,w)+\rmT(w,N\mathbf{e}_1)-({\rm T}^*(0,v)+{\rm T}^*(v,w)+{\rm T}^*(w,N\mathbf{e}_1))\\
        &\geq \rmT(v,w)-{\rm T}^*(v,w)\\
        &\geq{} \underline{\tau}|v-w|_1+\delta_{\ref{lem2}} n/2-\underline{\tau}|v-w|_1-8d Mn_1\\
  &\geq \delta_{\ref{lem2}} n/4.
  \end{split}
\end{equation}
 In particular, \eqref{diff} follows.
\end{proof}
By \eqref{oppappi}, if we take $c$ sufficiently small depending on $s$ again, then for $N$ large enough,
\begin{equation*}
  \begin{split}
    \P(\tau^*\in A_1)&\geq \left\{\alpha{}e^{-\beta{}c^{r}(\log{N})^{\frac{r}{r+1}}}\right\}^{|\tilde{\rm E}|}\P\left(\tau^*_e<\underline{\tau}+c\right)^{|\tilde{\rm C}|}\\
&\geq{}N^{-\delta}.
  \end{split}
\end{equation*}
Thus we have
\begin{equation}\label{combining dayo}
\begin{split}
\P(A)\geq N^{-\delta}\P\left(\text{${\rm B}^j(l;n)$ is black for $\tau$ and $\exists\gamma\in\O_N$ crosses ${\rm B}^j(l;n)$}\right).
  \end{split}
\end{equation}
Combined with Proposition~\ref{crucial-prop},  this proves Proposition~\ref{crucial-don}.
\subsection{Lower bound for $r=0$}\label{low:r=0}
We suppose $r=0$, where
${\rm f}_{d,r}(N)=\frac{\log{N}}{\log{\log{N}}}.$ Let $n= [s{\rm f}_{d,0}(N)]$ and $n_1=[s^{1+\frac{1}{2d}}{\rm f}_{d,0}(N)],$ and the other definitions be the same as before. Then the same arguments as in subsection \ref{low:r<d-1} work to prove  Proposition \ref{crucial-prop}. Moreover since $\sharp \tilde{\rm E}^j(l;n)\leq 2d(3n/n_1)^{d-1}(3n)\le s^{1/2}{\rm f}_{d,0}(N)$,  for any $\delta>0$, if $s$ is sufficiently small, then the probability of $\{\tau^*\in A_1\}$ can be bounded from below by
\begin{equation}
\begin{split}
&\left(\left(\alpha{}\frac{c\log{N}}{\log{\log{N}}}\right)^{-\beta}\right)^{|\tilde{\rm E}|}\P\left(\tau^*_e<\underline{\tau}+c\right)^{|\tilde{\rm C}|}\geq{}N^{-\delta}.
  \end{split}
\end{equation}
The rest is the same as before.
\subsection{Lower bound for $r>d-1$}\label{LB2}
We suppose $r>d-1$, where
\begin{equation}\label{order}
  {\rm f}_{d,r}(N)=\begin{cases}
  (\log{N})^{\frac{1}{d}} , &\text{ if }d-1< r< d,\\
      (\log{N})^{\frac{1}{d}}(\log{\log{N}})^{-\frac{1}{d}}, &\text{ if }r= d,\\
  (\log{N})^{\frac{1}{r}} , &\text{ if }d<r.\\
  \end{cases}
\end{equation}
We take $M>0$ sufficiently large and  $s>0$ sufficiently small depending on $M>0$ specified later. Set
$$n=\lf s{\rm f}_{d,r}(N)\rf\text{ and }n_1=2\lf s^2 {\rm f}_{d,r}(N)\rf.$$ Here we have defined $n_1$ to be even so that $n_1/2$ is an integer. We use the same definitions of ${\rm C}^j(l;n),$ $\tilde{\rm C}^j(l;n)$, and ${\rm D}^j(l;n)$ as before. We change the definitions of ${\rm E}^j(l;n)$ and $\tilde{\rm E}^j(l;n)$ ( see Figure~\ref{fig:four}) as
    \al{
{\rm E}^j(l;n)&=\{v\in{}{\rm C}^j(l;n)|~\exists w\in {\rm D}^j(l;n)\text{ s.t. }|v-w|_1=n_1/2\},\\
    \tilde{\rm E}^j(l;n)&=\{\langle v,w \rangle|~v\in {\rm E}^j(l;n),w=v+\mathbf{e}_i\in{}{\rm C}^j(l;n)\text{, $ \exists i\in\{1\cdots,d\}$}\}.
}
    Given $a\in {\rm C}^j(l;n)\backslash {\rm E}^j(l;n)$, let $W_a$ be the connected component of ${\rm C}^j(l;n)\backslash {\rm E}^j(l;n)$ containing $a$, i.e. $W_a={\rm Conn}(a,{\rm C}^j(l;n)\backslash {\rm E}^j(l;n))$ from the notation in Section~\ref{section: notation}. We list the basic properties of $W_a$ and ${\rm C}^j(l;n)$.
\begin{figure}[b]
  \includegraphics[width=7.0cm]{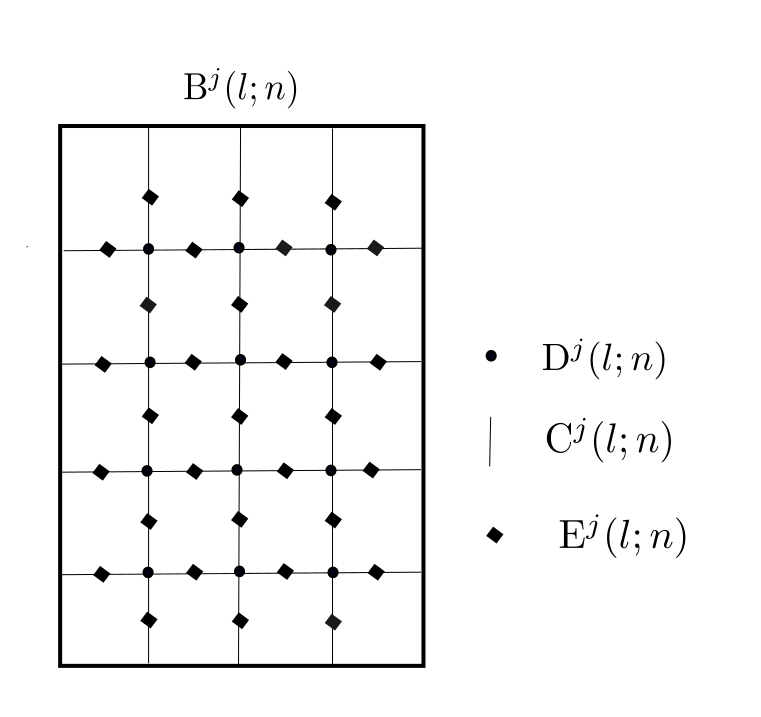}
  \hspace{4mm}
  \includegraphics[width=7.0cm]{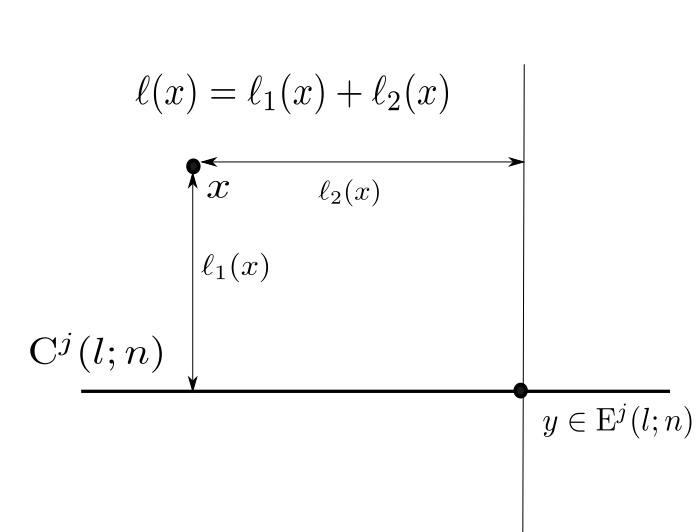}
\caption{}
Left: The figure of ${\rm C}^j(l;n)$, ${\rm D}^j(l;n)$, ${\rm E}^j(l;n)$  for $r>d-1$.\\
Right: The figure of $\ell(x)$, $\ell_1(x)$ and $\ell_2(x)$.
  \label{fig:four}
\end{figure}
    \begin{lem}\label{lem11}
      ~\\
      \begin{enumerate}[(i)]
      \item  For any $a\in {\rm C}^j(l;n)\backslash {\rm E}^j(l;n)$ and $b\in W_a$, there exists a path $\pi=(x_0,\cdots,x_l)$ from $a$ to $b$ which lies only on ${\rm C}^j(l;n)\backslash {\rm E}^j(l;n)$ and $l=|a-b|_1$.\label{prop-Wa}
      \item  For any $a,b\in {\rm C}^j(l;n)$ with $|a-b|_1<n_1/4$, there exists a path $\pi=(x_0,\cdots,x_l)$ from $a$ to $b$ which lies only on ${\rm C}^j(l;n)$ and $l=|a-b|_1$. \label{prop-Cj}
      \item   For any $a\in \partial {\rm B}^j(l;n)$ and $b\in  {\rm C}^j(l;n)$ with $|a-b|_1<n_1/4$, there exists a path $\pi=(x_0,\cdots,x_l)$ from $a$ to $b$ which lies only on ${\rm C}^j(l;n)\cup \partial {\rm B}^j(l;n)$ and $l=|a-b|_1$.
      \item If $|a-b|_1\le n_1/4$ and $W_{a}\neq W_{b}$, then there exists $\langle y_1,y_2\rangle \in \tilde{\rm E}^j(l;n)$ such that $|a-y_1|_1,|b-y_1|_1\leq n_1/4+1$ and a line $L$ including both $y_1$ and $y_2$ also includes both $a$ and $b$.
        \item For any $a\in \partial{\rm B}^j(l;n)$, $\{b\in {\rm C}^j(l;n)|~|b-a|_1<  n_1/2\}$ is a straight line.
        \end{enumerate}
    \end{lem}
    \begin{proof}
      (i) It is easy to see that for any connected component, namely $W_a$, there exists $x\in ({\rm D}^j(l;n)\cup\partial {\rm B}^j(l;n))\cap W_a$  such that
      $$W_a\subset \{x+k\mathbf{e}_i|~ i\in\{1,\cdots,d\},~k\in\Z^d\}.$$
      This yields (i).\\

      (ii) Fix $a\in {\rm C}^j(l;n)$. If there exists $x\in {\rm D}^j(l;n)$ such that $|a-x|_\infty<n_1/4$, then since
      $$\{y\in {\rm C}^j(l;n)|~|a-y|_1<n_1/4\}\subset \{x+k\mathbf{e}_i|~ i\in\{1,\cdots,d\},~k\in\Z^d\},$$
      the claim holds. Otherwise, 
      $\{y\in {\rm C}^j(l;n)|~|a-y|_1<n_1/4\}$ is a subset of a straight line.
      Since $\{y\in {\rm C}^j(l;n)|~|a-y|_1<n_1/4\}$ is connected, (ii) holds.\\

      (iii),(iv),(v) They follow from the construction of ${\rm C}^j(l;n)$ directly. 
      \end{proof}

    Given $x\in\Z^d$, we define $\ell(x)={\rm d}_1(x,{\rm E}^j(l;n))$ and given $\langle x,y\rangle\in {\rm E}^d$, $\ell(\langle x,y\rangle)=\inf\{\ell(x),\ell(y)\}$.

    \begin{lem}\label{number-est}
      There exists $C(s)>0$ such that for any $0\leq \ell\leq 2dn_1$,
      \begin{equation}
        \sharp\{e\in {\rm E}^d|~\ell(e)=\ell\}\leq  C(s)(\ell+1)^{d-1}.
      \end{equation}
      If $\ell>2dn_1$, then
      \begin{equation}\label{2-0}
        \sharp\{e\in {\rm E}^d|~\ell(e)=\ell,~e\subset {\rm B}^j(l;n)\}=0.
      \end{equation}
    \end{lem}
    \begin{proof}
 We begin with the case $\ell\leq 2d n_1$.    Since there exists $C_1(s)>0$ such that $\sharp {\rm E}^j(l;n)\leq C_1(s)$, we have
    \begin{equation*}
      \begin{split}
      \sharp\{x\in\Z^d|~\ell(x)=\ell\}&\leq \sharp \{x\in\Z^d|~\exists y\in {\rm E}^j(l;n)\text{ s.t. }|x-y|_1=\ell\}\\
      &\leq C_1(s)\sharp\{x\in \Z^d|~|x|_1=\ell\}.
      \end{split}
    \end{equation*}
    By $\sharp\{x\in \Z^d|~|x|_1=\ell\}\leq d^d(\ell+1)^{d-1}$, we have $$\sharp\{x\in\Z^d|~\ell(x)=\ell\}\leq d^d C_1(s)(\ell+1)^{d-1}.$$
    In particular, we get $\sharp\{e\in {\rm E}^d|~\ell(e)=\ell\}\leq  2d^{d+1}C_1(s)(\ell+1)^{d-1}$ as desired. If $ \ell>2dn_1$, then \eqref{2-0} is trivial due to the way of construction.

    \end{proof}



 Unlike the case $r<d-1$, when $r>d-1$, $\P(\tau\in A_1)$ decays faster than any polynomial, which implies that the lower bound from $A_1$--condition is not appropriate in this case. Hence, we need to consider a different  condition.
\begin{Def}\label{Def:A2}
 We say that the collection $\tau=\{\tau_e\}_{e\in {\rm E}^d}$ satisfies $A_2$--condition if 
 $$ \begin{cases}
  c^2{\rm f}_{d,r}(N)\leq{}\tau_e\leq{}\kappa{}c^2{\rm f}_{d,r}(N), & \text{if
$e\in{}\tilde{\rm E}^j(l;n)$}, \\
  \tau_e\leq{}\underline{\tau}+c^2, &\text{if
$e\in\tilde{\rm C}^j(l;n)\backslash{}\tilde{\rm E}^j(l;n)$},\\
  \tau_e\geq{}\frac{c{\rm f}_{d,r}(N)}{\ell(e)+1}\lor{}M^2, &\text{if
$e\subset{}\iota({\rm B}^j(l;n))$ and $e\notin\tilde{\rm C}^j(l;n)$},\\
  \tau_e\geq{}(\log{N})^{\frac{1}{2d(r+1)}},& \text{if $e\cap\partial{}{\rm B}^j(l;n)\neq{}\emptyset$, $e\cap \iota({\rm B}^j(l;n))\neq{}\emptyset$, $e\notin\tilde{\rm C}^j(l;n)$,}
  \end{cases}$$
where recall that $\iota(B)=B\backslash \partial B$. If $\tau$ satisifes $A_2$--condition, we write $\tau\in A_2$.
\end{Def}
 Since an edge weight on $\iota({\rm B}^j(l;n))\backslash {\rm C}^j(l;n)$ is high under $A_2$--condition, an optimal path is more likely to lie on ${\rm C}^j(l;n)$ while crossing ${\rm B}^j(l;n)$.

A resampled configuration $\tau^*=\{\tau^*_e\}_{e_\in {\rm E}^d}$ is taken to be $\tau^*_e=\tau_e$ if $e\cap \iota({\rm B}^j(l;n))= \emptyset$  and an independent copy of $\tau_e$ if $e\cap \iota({\rm B}^j(l;n))\neq \emptyset$. We define the event $A$ as
\begin{equation}\label{A1'}
\begin{split}
  A=\{\tau^*\in A_2\}\cap \{
  {\rm B}^j(l;n)\text{ is black for $\tau$ and }\exists\gamma\in \O_N\text{ crosses }{\rm B}^j(l;n)\}.
\end{split}
\end{equation}
\begin{prop}\label{prob A2}
  For any $\delta>0$, if $s>0$ is sufficiently small depending on $M$ and $c>0$ is sufficiently small depending on $s$, then
    $$\P\left(\tau\in A_2\right)\geq{}N^{-\delta}.$$
    \end{prop}
\begin{proof}
  By the fact $\P\left(\tau_e>a \lor b\right)\ge \P\left(\tau_e>a\right)\P\left(\tau_e>b\right)$ for $a,b\ge 0$ and Lemma~\ref{number-est}, we have
    \begin{equation*}
\begin{split}
  \P\left(\tau\in A_2\right) &\geq{}\left(\alpha e^{-\beta
c^{2r}{\rm f}_{d,r}(N)^r}\right)^{|\tilde{\rm E}|}\P\left(\tau_e\le
  \underline{\tau}+c^2\right)^{|\tilde{\rm C}|}(1\land{}\alpha)^{2d|{\rm B}|}\prod^{2dn_1}_{k=1}\left(e^{-\beta\left(\frac{c{\rm f}_{d,r}(N)}{k}\right)^r}\right)^{C(s)k^{d-1}}\\
  &\hspace{20mm}\cdot\P\left(\tau_e>M^2\right)^{2d|{\rm B}|}
  \left(\alpha e^{-\beta (\log{N})^{\frac{1}{2d}}}\right)^{2d|\partial{}{\rm B}|}\\
  &\ge \exp{\left(-\frac{\delta\log{N}}{2}\right)}\exp{\left(-\beta C(s)\sum^{2dn_1+1}_{k=1}c^r\frac{{\rm f}_{d,r}(N)^r}{k^{r-d+1}}\right)}\\
  &\ge \exp{(-\delta\log{N})}\ge N^{-\delta},
  \end{split}
\end{equation*}
    where we have used $({\rm f}_{d,r}(N))^d,\,(\log{N})^{1/2d}({\rm f}_{d,r}(N))^{d-1}\leq \log{N}$ since $r>d-1$.
    \end{proof}
\begin{prop}\label{crucial-prop2}
  For $M$ large enough, if $s$ is sufficiently small and $c$ is sufficiently small depending on $s$, then for $N$ large enough, unless $0\in {\rm B}^j(l;n)$ or $N\mathbf{e}_1\in {\rm B}^j(l;n)$, on the event $A$, any $\pi^*\in\O^*_N$
does not touch ${\rm B}^j(l;n)\backslash(\partial{}{\rm B}^j(l;n)\cup{}{\rm C}^j(l;n))$
and passes on an edge of $\tilde{\rm E}^j(l;n)$ at least one time.
    \end{prop}

    \begin{proof}
      By the same way as in the case $r<d-1$, we have
      $\rmT_N>{\rm T}^*_N$, and  any $\pi^*\in\O^*_N$ has to
enter inside ${\rm B}^j(l;n)$. We take an arbitrary optimal path $\pi^*\in \O^*_N$ and write $\pi^*=\{x_1,\cdots,x_{K}\}$. By assuming that there exists $k\in\{1,\cdots,K\}$ such that $x_k\in ({\rm B}^j(l;n)\backslash(\partial{}{\rm B}^j(l;n)\cup{}{\rm C}^j(l;n)))$, we shall derive a contradiction. 

We define $p=\max\{l\le k|~x_l\in \partial{}{\rm B}^j(l;n)\cup{}{\rm C}^j(l;n)\}$
and
$q=\min\{l\ge k|~x_l\in \partial{}{\rm B}^j(l;n)\cup{}{\rm C}^j(l;n)\}$.
Note that $q-p\ge 1$. Set $a=x_p$ and $b=x_q$. Define $C(s)>0$ so that $\sharp {\rm E}^j(l;n)\leq C(s)$.\\

\noindent\underline{Step 1} ($a,b \in {\rm C}^j(l;n)$): Suppose $a\in\partial{}{\rm B}^j(l;n)$ and we shall derive a contradiction. Since $\pi^*|_{a\to b}$ has to pass on an edge whose weight is at least $(\log{N})^{{\frac{1}{2d(r+1)}}}$ and
passes only on ${\rm B}^j(l;n)\backslash(\partial {\rm B}^j(l;n)\cup {\rm C}^j(l;n))$
except for the starting and ending points, we have
\begin{equation}\label{6.5}
{\rm T}^*(a,b)\ge (\log{N})^{{\frac{1}{2d(r+1)}}}+(|a-b|_1-1)M^2.
\end{equation}
 If $|a-b|_1 \le n_1/4$, by Lemma \ref{lem11}--(iii) and (Black-2), then there exists a path $\gamma:a\to{}b$ on
$\partial{}{\rm B}^j(l;n)\cup{}{\rm C}^j(l;n)$ such that
 $${\rm T}^*(\gamma)<M\left((\log{N})^{\frac{1}{4d(r+1)}}+|a-b|_1\right)<{\rm T}^*(a,b),$$
 which is  a contradiction.

 On the other hand, if $|a-b|>n_1/4$, by $A_2$--condition, then for $c$ small enough, we can take a path $\gamma:a\to{}b$ on $\partial{}{\rm B}^j(l;n)\cup{}{\rm C}^j(l;n)$ 
so that
\begin{equation*}
  \begin{split}
    {\rm T}^*(\gamma)&<M|a-b|_1+C(s)c^2\kappa {\rm f}_{d,r}(N)\\
    &<(\log{N})^{{\frac{1}{2d(r+1)}}}+(|a-b|_1-1)M^2\\
    &\leq {\rm T}^*(a,b),
  \end{split}
\end{equation*}
which is also a contradiction. Thus, we have $a\in{}{\rm C}^j(l;n)$. Similarly, we get $b\in{}{\rm C}^j(l;n)$. \\

\noindent\underline{Step 2} ($|a-b|_1\leq n_1/4$):
Note that by the same reason of \eqref{6.5},
${\rm T}^*(a,b)\ge |a-b|_1 M^2.$ Take a path $\gamma:a\to{}b$ on
${\rm C}^j(l;n)$ such that
${\rm T}^*(\gamma)\leq{}(\underline{\tau}+c^2)(|a-b|_1+4n_1)+C(s)\kappa c^2{\rm f}_{d,r}(N)$. It follows that
\begin{equation*}
  \begin{split}
    M^2|a-b|_1&\leq{}{\rm T}^*(a,b)\\
    &=\sum^{q-1}_{i=p}\tau_{\langle x_i,x_{i+1} \rangle}\\
    &\le (\underline{\tau}+c^2)(|a-b|_1+4n_1)+C(s)\kappa c^2{\rm f}_{d,r}(N),    
  \end{split}
\end{equation*}

which leads to $$|a-b|_1\leq \frac{4(\underline{\tau}+c^2)n_1+C(s)\kappa c^2{\rm f}_{d,r}(N)}{(M^2-\underline{\tau}-c^2)n_1}n_1.$$
If we take $M>0$ sufficiently large, then it follows that $|a-b|_1\leq n_1/4$.\\

\noindent\underline{Step 3} ($W_a\neq W_b$):
When $a$ and $b$ both belong to the same connected component, i.e. $W_a=W_b$, by Lemma \ref{lem11}--(i), we can take a path $\gamma:a\to{}b$
on ${\rm C}^j(l;n)$ such that
\begin{equation*}
  \begin{split}
    {\rm T}^*(\gamma)&\leq{}(\underline{\tau}+c^2)|a-b|_1\\
    &<M^2|a-b|_1\leq{}{\rm T}^*(a,b),
     \end{split}
\end{equation*}
which is also a contradiction. Thus $W_a\neq W_b$.\\

\noindent\underline{Step 4} (Conclusion):
It follows from Lemma \ref{lem11}--(ii) that we can take a path $\gamma:a\to{}b$ on ${\rm C}^j(l;n)$ such that
\ben{\label{base ineq}
  {\rm T}^*(a,b)\leq {\rm T}^*(\gamma)\leq{}c^2\kappa {\rm f}_{d,r}(N)+(\underline{\tau}+c^2)|a-b|_1.
  }
By Lemma \ref{lem11}--(iv), the line between $a$ and $b$ lies on ${\rm C}^j(l;n)$ and it includes exactly one vertex of ${\rm E}^j(l;n)$ between $a$ and $b$, say $x$. If there exists $k\in \Iintv{p, q}$ such that $x_k\in\partial {\rm B}^j(l;n)$, then  $\max\{|a-x_k|_1,|b-x_k|_1\}\ge n_1/4$ and by \eqref{base ineq},

\begin{equation*}
  \begin{split}
    M^2n_1/4&\leq {\rm T}^*(a,b)\\
    &<c^2\kappa {\rm f}_{d,r}(N) +(\underline{\tau}+c^2)|a-b|_1\\
    &\le c^2\kappa {\rm f}_{d,r}(N) +(\underline{\tau}+c^2)n_1/4,
\end{split}
  \end{equation*}
which contradicts that $M$ is sufficiently large. Thus for any $k\in \Iintv{p, q}$, $x_k\in \iota({\rm B}^j(l;n))$. 

Since $$|x_k-x|_1\leq |x_k-a|_1+|x_k-b|_1\leq q-p\text{ for any $p<k<q$},$$
we have $\ell(\langle x_k,x_{k+1} \rangle)\leq q-p$. This yields  
\begin{equation*}
  \begin{split}
    {\rm T}^*(a,b)&\leq (\underline{\tau}+c^2)|a-b|_1+c^2\kappa {\rm f}_{d,r}(N)\\
    &< M^2|a-b|_1 \lor\sum^{q-1}_{i=p}\frac{c{\rm f}_{d,r}(N)}{q-p+1}\\
    &\leq{}\sum^{q-1}_{i=p}\left(M^2\lor \frac{c{\rm f}_{d,r}(N)}{\ell(\langle x_i,x_{i+1} \rangle)+1}\right)\\
    &\leq \sum^{q-1}_{i=p}\tau^*_{\langle x_i,x_{i+1} \rangle}={\rm T}^*(a,b),
  \end{split}
  \end{equation*}
where we have used \eqref{base ineq} in the first inequality, $q-p\geq |a-b|_1$ and $ \ell(\langle x_k,x_{k+1} \rangle)\leq q-p$ in the third inequality, and $A_2$--condition in the last inequality. It leads to a contradiction. Therefore, we conclude that such $x_k$ does not exist and any  $\pi^*\in \O_N^*$ does not touch ${\rm B}^j(l;n)\backslash(\partial{}{\rm B}^j(l;n)\cup{}{\rm C}^j(l;n))$.\\

Finally, we show that any optimal path passes through an edge of $\tilde{\rm E}^j(l;n)$. Since any optimal path $\pi^*\in \O_N^*$ does not touch
${\rm B}^j(l;n)\backslash(\partial{}{\rm B}^j(l;n)\cup{}{\rm C}^j(l;n))$, if $\pi^*$ does not touch $\{x\in {\rm C}^j(l;n)|~x\not\sim_{{\rm C}^j(l;n)\backslash {\rm E}^j(l;n)} \partial {\rm B}^j(l;n)\}$, then $\pi^*\cap \iota({\rm B}^j(l;n))^c$ is also a path and we can take an optimal path so that it does not enter inside ${\rm B}^j(l;n)$, which contradicts $\rmT_N>\rmT_N^*$ mentioned at the beginning of the proof. Thus, $\pi^*$ touches $\{x\in {\rm C}^j(l;n)|~x\not\sim_{{\rm C}^j(l;n)\backslash {\rm E}^j(l;n)} \partial {\rm B}^j(l;n)\}$  and needs to pass on at least one edge of 
$\tilde{\rm E}^j(l;n)$.
    \end{proof}
 Proposition~\ref{crucial-prop2}   implies that on the event $A$, ${\rm B}^j(l;n)$ is good for $\tau^*$. Hence, we obtain \eqref{crucial}.
 Together with Proposition~\ref{prob A2}, as in \eqref{combining dayo}, we conclude Proposition~\ref{crucial-don}.
    \subsection{Lower bound for $r=d-1$ with $d\geq{}3$}
We consider the case $r=d-1$ with $d\geq 3$, where ${\rm f}_{d,d-1}(N)=(\log{N})^{\frac{1}{d}}(\log{\log{N}})^{\frac{d-2}{d}}$.    We take $s>0$ and $M\in\N$ to be chosen later. We define $n$, $n_1$ as in Section~\ref{LB2}. We use the same notation as in subsection \ref{LB2}.   
    
    We define (See Figure~\ref{fig:ell})
    $${\rm F}^j(l;n)=\{v\in {\rm B}^j(l;n)\backslash(\partial{}{\rm B}^j(l;n)\cup{}{\rm C}^j(l;n))|~ \exists w\in {\rm C}^j(l;n)\text{ s.t. }|v-w|_1<{}(\log{N})^{\frac{1}{4d^2}}\}.$$
    Given $x\in \Z^d$, we define $\ell_1(x)$ and $\ell_2(x)$ as follows (See Figure~\ref{fig:four}):
    $$\ell_1(x)= {\rm d}_1(x,{\rm C}^j(l;n)),~\ell_2(x)={\rm d}_1(x,{\rm E}^j(l;n))_1-\ell_1(x).$$
    Given an edge $e=\langle x, y\rangle$, we define $\ell_1(e)$ and $\ell_2(e)$ as  $$\ell_1(e)=\ell_1(x)\land \ell_1(y),~\ell_2(e)=\ell_2(x)\land \ell_2(y).$$
    
    \begin{figure}[b]
      \includegraphics[width=7.3cm]{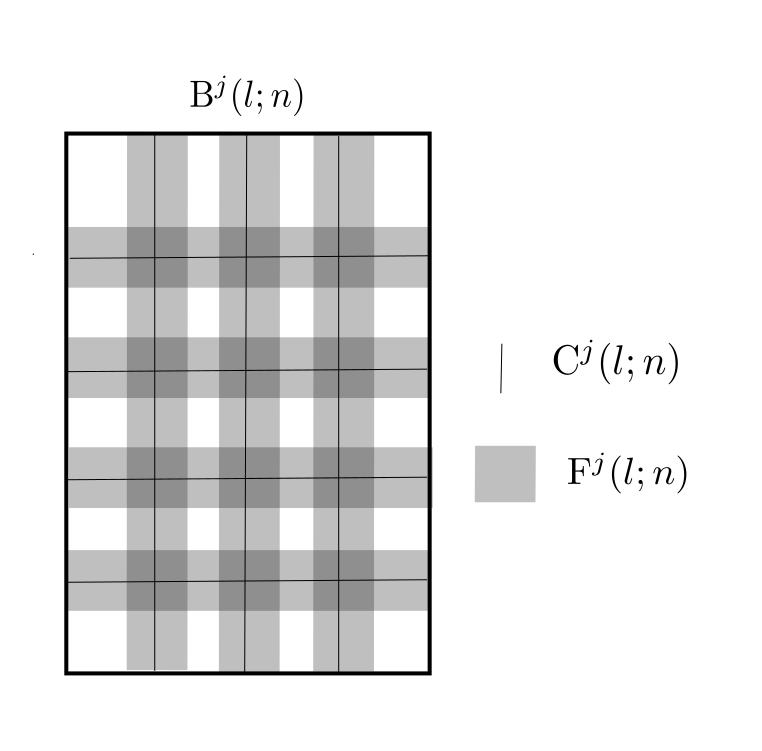}\hspace{4mm}
       \includegraphics[width=7.0cm]{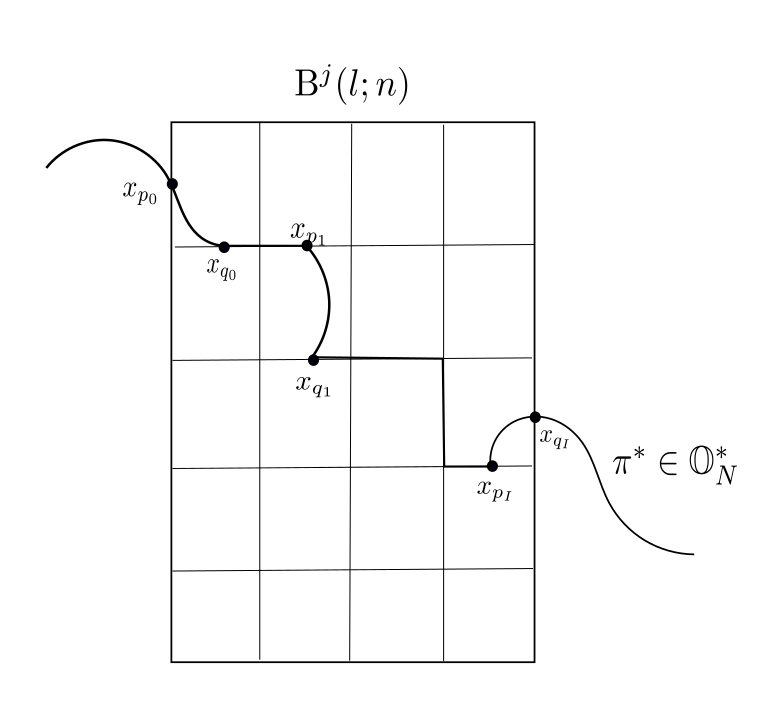}
       \caption{}
       Left: The figure of ${\rm F}^j(l;n)$.\\
Right:       The image of the proof of Proposition \ref{crucial-prop3}.

  \label{fig:ell}
\end{figure}
\begin{lem}\label{estimate-F}
  There exists a positive constant $C(s)$ such that if $0\le \ell\leq{}(\log{N})^{\frac{1}{4d^2}}$ and $0\le k \leq 2n_1$, then
  \begin{equation}\label{121}
    \sharp\{e\in {\rm E}^d|~e\cap {\rm F}^j(l;n)\neq\emptyset,~\ell_1(e)=\ell,\,\ell_2(e)=k\}\leq C(s)(\ell+1)^{d-2}.
     \end{equation}
  Otherwise, if $\ell>(\log{N})^{\frac{1}{4d^2}}$ or $k> 2n_1$, then
  \begin{equation}\label{122}
    \sharp\{e\in {\rm E}^d|~e\cap {\rm F}^j(l;n)\neq\emptyset,~\ell_1(e)=\ell,\,\ell_2(e)=k\}=0.
    \end{equation}
\end{lem}
\begin{proof}
  Let $L=\{k\mathbf{e}_1|~k\in\Z\}$.   The inequality \eqref{121} follows from 
  \begin{equation*}\begin{split}
    \sharp\{e\in {\rm E}^d|~\ell_2(e)=k,~\ell_1(e)=\ell\}&\leq 2d (\sharp {\rm E}^j(l;n))\sharp\{v\in\Z^d|~|v_1|=k,~|v|^{(2)}_1=\ell\}\\
    &\leq C(s)(\ell+1)^{d-2},
    \end{split}\end{equation*} 
  where we define $|v|^{(2)}_1=|v|_1-|v_1|$.
 By the definition of ${\rm F}^j(l;n)$, \eqref{122} is trivial.
  \end{proof}
When $r>d-1$, we do subject conditions on all weights in ${\rm B}^j(l;n)$. However, when $r=d-1$, since $({\rm f}_{d,r}(N))^d\gg (\log{N})$, we cannot do that. Hence, we need to consider an even weaker condition.  This makes the proof more complicated.
\begin{Def}\label{Def:A3}
We say that the collection $\tau=\{\tau_e\}_{e\in {\rm E}^d}$ satisfies $A_3$--condition if 
 \begin{equation}\label{A3-cond}
    \begin{cases}
  c^2{\rm f}_{d,d-1}(N)\leq{}\tau_e\leq{}\kappa{}c^2{\rm f}_{d,d-1}(N), & \text{if $e\in{}\tilde{\rm E}^j(l;n)$} \\
  \tau_e\leq{}\underline{\tau}+c^2, &\text{if $e\in\tilde{\rm C}^j(l;n)\backslash{}\tilde{\rm E}^j(l;n)$}\\
  \tau_e\geq{}\frac{c{\rm f}_{d,d-1}(N)}{(\ell_1(e)+1)\log{(\ell_2(e)+2)}}&\text{if $e\cap{}F_j(l;n)\neq\emptyset$ and $e\cap{}\partial{}B_j(l;n)=\emptyset$}\\
  \tau_e\geq{}(\log{N})^{\frac{1}{2d^2}},& \text{if $e\cap\partial{}{\rm B}^j(l;n)\neq{}\emptyset$, $e\not\subset\partial{}{\rm B}^j(l;n))$, $e\notin\tilde{\rm C}^j(l;n)$.}
    \end{cases}
 \end{equation}
  If $\tau$ satisifes $A_3$--condition, then we write $\tau\in A_3$.
\end{Def}
We note that under $A_3$--condition, for $e\in {\rm E}^d$ with $e\cap{}F_j(l;n)\neq\emptyset$ and $e\cap{}\partial{}B_j(l;n)=\emptyset$, $\tau_e\geq M^2$ for $N$ large enough. A resampled configuration $\tau^*=\{\tau^*_e\}_{e_\in {\rm E}^d}$ is taken to be an independent copy if $e$ satisfies one of the conditions which appear in \eqref{A3-cond}  and  $\tau_e^*=\tau_e$ otherwise.

We define the event $A$ as
  \begin{equation}\label{A3'}
\begin{split}
  A=\{\tau^*\in A_3\}\cap \{
  {\rm B}^j(l;n)\text{ is black for $\tau$ and }\exists\gamma\in \O_N\text{ crosses }{\rm B}^j(l;n)\}.
\end{split}
  \end{equation}
 On the event $A$, if $e\notin \tilde{\rm C}^j(l;n)$, then $\tau_e^*\geq\tau_e$. Thus, the following proposition follows.
  \begin{prop}\label{A3-tilde}
    On the event $A$, the following holds.  For any $v,w\in \iota({\rm B}^j(l;n))$ with $|v-w|_1\ge (\log{}N)^{\frac{1}{4d^2}}$ and a path $\pi: v \to w$ satisfying $\pi\cap \tilde{\rm C}^j(l;n)=\emptyset$,
  \begin{equation}\label{A3-cond2}
    \begin{split}
       {\rm T}^*(\pi)\geq{}(\underline{\tau}+\delta_{\ref{lem2}}){}|v-w|_1.
      \end{split}
    \end{equation}
    \end{prop}
  \begin{prop}\label{prob A3}
  For any $\delta>0$, there exists $s_0>0$ such that for any $s<s_0$, there exists $c_0(s)>0$ such that for any $c<c_0(s)$,
    $$\P\left(\tau^*\in A_3\right)\geq{}N^{-\delta}.$$
    \end{prop}
  \begin{proof} 
  If we take $c>0$ sufficiently small depending on $s>0$, by  Lemma~\ref{estimate-F}, then
    \begin{equation*}
\begin{split}
  \P\left(\tau^*\in A_3\right)&\geq{}\left(\alpha e^{-\beta c^{2(d-1)}{\rm f}_{d,d-1}(N)^{d-1}}\right)^{|\tilde{\rm E}|}\P\left(\tau_e\le \underline{\tau}+c^2\right)^{|\tilde{\rm C}|}\\
  &\hspace{4mm}\times \prod^{\lf (\log{N})^{\frac{1}{4d^2}}\rf}_{\ell=0}\prod^{2n_1}_{k=0}\left(\alpha e^{-\beta\left(\frac{c{\rm f}_{d,d-1}(N)}{(\ell+1)\log{(k+2)}}\right)^{d-1}}\right)^{C(s)(\ell+1)^{d-2}}\left(\alpha e^{-\beta (\log{N})^{\frac{1}{2d}}}\right)^{2d|\partial{}{\rm B}|}\\
  &\ge \exp{\left(-\frac{\delta \log{N}}{2}\right)}\exp{\left(-\beta   C(s)\sum^{\lf (\log{N})^{\frac{1}{4d^2}}\rf}_{\ell=0}\sum^{2n_1}_{k=0}c^{d-1}\frac{{\rm f}_{d,d-1}(N)^{d-1}}{(\ell+1)(\log{(k+2)})^{d-1}}\right)}\\
  &\geq \exp{(-\delta\log{N})}=N^{-\delta},
  \end{split}
\end{equation*}
    where  we have used the following facts  in the last inequality that for $s$ small  enough,
    $$\sum^{2n_1}_{k=0}\frac{1}{(\log{(k+2)})^{d-1}}\leq{}\frac{{\rm f}_{d,d-1}(N)}{(\log{\log{N}})^{d-1}},$$
    $$\sum^{\lf (\log{N})^{\frac{1}{4d^2}}\rf}_{\ell=0} \frac{1}{\ell+1}\leq \log{\log{N}}.$$
  \end{proof}

    \begin{prop}\label{crucial-prop3}
   If $s$ is sufficiently small depending on $M$ and $c$ is sufficiently small depending on $s$, then on the event $A$, unless $0\in {\rm B}^j(l;n)$ or $N\mathbf{e}_1\in {\rm B}^j(l;n)$, any optimal path $\pi^*\in\O^*_N$ needs to pass through at least one edge of $\tilde{\rm E}^j(l;n)$.
    \end{prop}
    \begin{proof}
      By the same argument as in \eqref{new-passage}, one can check that
      \begin{equation}\label{differ}
        \rmT_N-\frac{\delta_{\ref{lem2}}}{4}n>{\rm T}^*_N.\end{equation}
      Thus any optimal path $\pi^*\in\O^*_N$ has to enter inside ${\rm B}^j(l;n)$. We take $\pi^*\in\O^*_N$ which is a self avoiding path and write $\pi^*=\{x_0,\cdots,x_K\}$.
      We define sequences $(p_i,q_i)$ inductively as follows: let
      $$p_0=\min\{l\in\Iintv{0, K}|~x_l\in {\rm B}^j(l;n)\}\text{ and }q_0=\min \{l\in\Iintv{p_0+1, K}|~x_l\in {\rm C}^j(l;n)\}.$$
    If we could define $(p_i,q_i)^k_{i=0}$, then we define
      \al{     p_{i+1}&=\min\{l\in\Iintv{q_k+1,K}|~x_l\in \Z^d\backslash {\rm C}^j(l;n)\}-1,\\
        q_{i+1}&=\min\{l\in\Iintv{p_{i+1}+1,K}|~x_l\in {\rm C}^j(l;n)\}.
      }
      Let $I=\inf\{l|~q_l=\infty\}$. We redefine $q_{I}=\max\{1\le l\le K|~x_l\in {\rm B}^j(l;n)\}$. (See Figure~\ref{fig:ell}.) By the same argument as in \eqref{new-passage}, one can check that
      \ben{\label{ato de tukau}
        {\rm T}^*(x_{p_0},x_{q_I}) 
      \le (\underline{\tau}+c^2)|x_{p_0}-x_{q_I}|_1+8dMn_1.
      }

      Before going into the details, we explain the strategy here. When $r>d-1$, since all the weights in ${\rm B}^j(l;n)\backslash {\rm C}^j(l;n)$  are sufficiently large under the condition $A_2$, optimal paths from $x_{p_0}$ to $x_{q_I}$ lies only on ${\rm C}^j(l;n)$ as we have proved. On the other hand, when $r=d-1$, since we do not subject conditions on all of the weights in ${\rm B}^j(l;n)$, we no longer conclude that optimal paths do not touch ${\rm B}^j(l;n)\backslash ({\rm C}^j(l;n)\cup \partial {\rm B}^j(l;n))$. However, if an path from $x_{p_0}$ to $x_{q_I}$ enjoys passing through ${\rm B}^j(l;n)\backslash {\rm C}^j(l;n)$ very long time, then Proposition~\ref{A3-tilde} yields that this path is not optimal. Thus, any optimal paths from $x_{p_0}$ to $x_{q_I}$ mostly lies on ${\rm C}^j(l;n)$. Hence, any optimal path needs to pass through an edge of  $\tilde{\rm E}^j(l;n)$. Let us make the above heuristic  rigorous.\\

      \underline{Step 1} ($|x_{p_i}-x_{q_i}|_1>n_1/4$): We will show that for any $i\in \Iintv{1,I-1}$, $|x_{p_i}-x_{q_i}|_1>n_1/4$. Until Step 1 is completed, we assume that $|x_{p_i}-x_{q_i}|_1\le n_1/4$ and finally we shall derive a contradiction. By Lemma \ref{lem11}--(i), there exists a path $\pi^*:x_{p_i}\to x_{q_i}$ on ${\rm C}^j(l;n)$ such that
      \begin{equation}\label{basic}
      {\rm T}^*(x_{p_i},x_{q_i})  \leq {\rm T}^*(\pi^*)\leq (\underline{\tau}+c^2)|x_{p_i}-x_{q_i}|_1+c^2{\rm f}_{d,d-1}(N)\mathbf{1}_{\{W_{p_i}\neq W_{q_i}\}}.
      \end{equation}
      We divide into two cases.
      
        \underline{Case 1} ($W_{x_{p_i}}=W_{x_{q_i}}$):
  If $|x_{p_i}-x_{q_i}|_1\geq (\log{N})^{\frac{1}{4d^2}}$, by Proposition~\ref{A3-tilde}, ${\rm T}^*(x_{p_i},x_{q_i})\geq{}(\underline{\tau}+\delta_{\ref{lem2}})|x_{p_i}-x_{q_i}|_1,$ which contradicts \eqref{basic}.  Hence,  $|x_{p_i}-x_{q_i}|_1< (\log{N})^{\frac{1}{4d^2}}$.       Moreover, if there exists $k\in \Iintv{p_i,q_i}$ such that $|x_{p_i}-x_k|_1\ge (\log{N})^{\frac{1}{4d^2}}$, then Proposition~\ref{A3-tilde} yields 
      \begin{align*}
        {\rm T}^*(x_{p_i},x_{q_i})&\ge (\underline{\tau}+\delta_{\ref{lem2}})(\log{N})^{\frac{1}{4d^2}}\\
        &>(\underline{\tau}+c^2)|x_{p_i}-x_{q_i}|_1,
      \end{align*}
      which contradicts \eqref{basic}. Thus for any $k\in \Iintv{p_i,q_i}$, $|x_{p_i}-x_k|_1< (\log{N})^{\frac{1}{4d^2}}$.

      If, in addition, there exists $k\in \Iintv{p_i,q_i}$ such that $x_k\in\partial {\rm B}^j(l;n)$, then the path $\{x_{p_i},\cdots, x_{q_i}\}$ crosses $\partial {\rm  B}^j(l;n)$ without using $\tilde{\rm  C}^j(l;n)$. Hence, by $A_3$--condition,
      \begin{align*}
        {\rm T}^*(x_{p_i},x_{q_i})&\ge (\log{N})^{\frac{1}{2d^2}}\\
        &>(\underline{\tau}+c^2)(\log{N})^{\frac{1}{4d^2}} \\
        &\geq (\underline{\tau}+c^2)|x_{p_i}-x_{q_i}|_1,
      \end{align*}
      which contradicts \eqref{basic}. Hence $x_k\in \iota({\rm B}^j(l;n))$ and, since $|x_{p_i}-x_k|_1< (\log{N})^{\frac{1}{4d^2}}$ as proved before, $x_k\in {\rm F}^j(l;n)$ for any $p_i<k<q_i$. 

      On the other hand,  since  $x_k\in {\rm F}^j(l;n)$ for any $p_i<k<q_i$, by the remark below Definition~\ref{Def:A3},   we have
      \begin{align*}
        {\rm T}^*(x_{p_i},x_{q_i})&\ge M^2|x_{p_i}-x_{q_i}|_1\\
        &>(\underline{\tau}+c^2)|x_{p_i}-x_{q_i}|_1,
        \end{align*}
      which contradicts \eqref{basic}.
      
      \underline{Case 2} ($W_{x_{p_i}}\neq W_{x_{q_i}}$): Suppose that $W_{x_{p_i}}\neq W_{x_{q_i}}$ holds. If there exists $k\in \Iintv{p_i,q_i}$  such that $x_k\in\partial {\rm B}^j(l;n)$, by Lemma \ref{lem11}--(iv), (v) and $|x_{p_i}-x_{q_i}|_1\le n_1/4$, then we have $\max\{|x_{p_i}-x_k|_1,|x_{q_i}-x_k|_1\}\ge n_1/4$. It follows from Proposition~\ref{A3-tilde} that
      \begin{align*}
        {\rm T}^*(x_{p_i},x_{q_i})&\ge (\underline{\tau}+\delta_{\ref{lem2}})n_1/4\\
        &>(\underline{\tau}+c^2)|x_{p_i}-x_{q_i}|_1+c^2{\rm f}_{d,d-1}(N),
        \end{align*}
      which contradicts \eqref{basic}. Hence, we have that for any $k\in \Iintv{p_i,q_i}$ , $x_{k}\in \iota({\rm B}^j(l;n))$.

      Next we suppose that  there exists $k_1\in \Iintv{p_i,q_i}$  such that $x_{k_1}\in {\rm B}^j(l;n)\backslash( {\rm F}^j(l;n)\cup {\rm C}^j(l;n))$. Since   $x_{k}\in \iota({\rm B}^j(l;n))$ for any $k\in \Iintv{p_i,q_i}$ and $\ell_1(x_{p_i})=0$, $\ell_2(x_k)\leq 2n_1$ 
    and  $\ell_1(x_k)\leq k-p_1.$     Since $|x_{p_i}-x_{k_1}|_1\ge (\log{N})^{\frac{1}{4d^2}}$, under $A_3$--condition, we have  that by Proposition~\ref{A3-tilde},
          \begin{equation*}
       \begin{split}
    {\rm T}^*(x_{p_i},x_{q_i})  
         &\geq{}(\underline{\tau}+\delta_{\ref{lem2}})|x_{p_i}-x_{q_i}|_1\lor\left(\sum^{\lf (\log{N})^{\frac{1}{4d^2}}\rf}_{\ell=0}\frac{c{\rm f}_{d,d-1}(N)}{(\ell+1)\log{(2n_1)}}\right)\\
         &> (\underline{\tau}+\delta_{\ref{lem2}})|x_{p_i}-x_{q_i}|_1\lor (c{\rm f}_{d,d-1}(N)/4d^2)\\
         &> (\underline{\tau}+c^2)|x_{p_i}-x_{q_i}|_1+c^2{\rm f}_{d,d-1}(N),
       \end{split}
          \end{equation*}
          which also contradicts \eqref{basic}. Thus for any $k\in \Iintv{p_i,q_i}$ , $x_k\in {\rm C}^j(l;n)\cup {\rm F}^j(l;n)$.  Due to the definition of $(p_i,q_i)$, $x_k\notin {\rm C}^j(l;n)$ for $p_i<k<q_i$, which yields $x_k\in  {\rm F}^j(l;n)$.

          Lemma \ref{lem11}--(iv) yields that there exists $k_1\in \Iintv{p_i,q_i}$  such that $\ell_2(x_{k_1})=0$, which implies $$\max\{\ell_2(x_k)|~k\in \Iintv{p_i,q_i}\}\le q_i-p_i.$$ It follows from $A_3$--condtion that
     \begin{equation*}
       \begin{split}
         {\rm T}^*(x_{p_i},x_{q_i}) 
         &\geq{}{} \left(\sum^{q_i-p_i-1}_{\ell=0}\frac{c{\rm f}_{d,d-1}(N)}{(\ell +1)\log{|q_i-p_i+2|}}\right)\\
         &>  c{\rm f}_{d,d-1}(N)/2.
       \end{split}
     \end{equation*}
     Moreover, by Proposition~\ref{A3-tilde}, if $|x_{p_i}-x_{q_i}|\geq (\log{N})^{\frac{1}{4d^2}}$,  then
     $${\rm T}^*(x_{p_i},x_{q_i})\geq (\underline{\tau}+\delta_{\ref{lem2}})|x_{p_i}-x_{q_i}|_1.$$
     Putting things together, we have
     \al{
       {\rm T}^*(x_{p_i},x_{q_i})&\geq (\underline{\tau}+\delta_{\ref{lem2}})|x_{p_i}-x_{q_i}|_1\lor \frac{ c{\rm f}_{d,d-1}(N)}{2}\\
       &> (\underline{\tau}+c^2)|x_{p_i}-x_{q_i}|_1+c^2{\rm f}_{d,d-1}(N),
       }
     which contradicts \eqref{basic}.

     In all cases, we derived contradictions. Hence, we have $|x_{q_i}-x_{p_i}|_1>n_1/4$.\\
  
     \underline{Step 2} ($|x_{q_I}-x_{p_0}|_1\geq \frac{\delta_{\ref{lem2}}{}n}{4M}$): 
 Since we only resample the edges  inside  ${\rm B}^j(l;n)$ and the paths $\pi^*|_{0\to x_{p_0}}$ and $\pi^*|_{0\to x_{q_I}}$ does not use such edges,
     \begin{equation*}
       \begin{split}
         &{\rm T}^*_N={\rm T}^*(0,x_{p_0})+{\rm T}^*(x_{p_0},x_{q_I})+{\rm T}^*(x_{q_I},N\mathbf{e}_1),\\
         &{\rm T}^*(0,x_{p_0})\geq \rmT(0,x_{p_0}),\\
         &{\rm T}^*(x_{q_I},N\mathbf{e}_1)\geq \rmT(x_{q_I},N\mathbf{e}_1).
          \end{split}
     \end{equation*}
     Moreover,  by the triangular inequality, we get
     $$\rmT_N\leq \rmT(0,x_{p_0})+\rmT(x_{p_0},x_{q_I})+\rmT(x_{q_I},N\mathbf{e}_1).$$
         Thus, by \eqref{differ}, we have
     \begin{equation*}
       \begin{split}
         \delta_{\ref{lem2}} n/4&\le \rmT_N -{\rm T}^*_N\\
         &\leq \rmT(0,x_{p_0})+\rmT(x_{p_0},x_{q_I})+\rmT(x_{q_I},N\mathbf{e}_1)-({\rm T}^*(0,x_{p_0})+{\rm T}^*(x_{p_0},x_{q_I})+{\rm T}^*(x_{q_I},N\mathbf{e}_1))\\
         &\leq \rmT(x_{p_0},x_{q_I})-{\rm T}^*(x_{p_0},x_{q_I})\leq \rmT(x_{p_0},x_{q_I}).
    \end{split}
     \end{equation*}
         By (Black-2), this is further bounded from above by $M(|x_{q_I}-x_{p_0}|_1\lor (\log{N})^{\frac{1}{4d^2}}).$  Therefore $\delta_{\ref{lem2}} n/4\leq M|x_{q_I}-x_{p_0}|_1$, which implies
     $$|x_{p_0}-x_{q_I}|_1\geq\frac{\delta_{\ref{lem2}}{}n}{4M}.$$
         \underline{Step 3} (Conclusion):          Note that
       \begin{equation}\label{diss}
         \begin{split}
           |x_{p_0}-x_{q_I}|_1 \le \sum^{I}_{i=0}|x_{q_i}-x_{p_i}|_1+\sum^{I-1}_{i=0}|x_{p_{i+1}}-x_{q_i}|_1.
       \end{split}
       \end{equation}
       By Step 1 and Proposition~\ref{A3-tilde}, for $i\in \Iintv{1,I-1}$,
       \begin{equation}\label{kkk1}
         (\underline{\tau}+\delta_{\ref{lem2}})|x_{q_i}-x_{p_i}|_1\leq {\rm T}^*(x_{p_i},x_{q_i}).
       \end{equation}
       Moreover, by Proposition~\ref{A3-tilde}, for $i=0$ or $i=I$, 
        \begin{equation}\label{kkk2}
         (\underline{\tau}+\delta_{\ref{lem2}})\left(|x_{q_i}-x_{p_i}|_1-(\log{N})^{\frac{1}{4d^2}}\right) \leq {\rm T}^*(x_{p_i},x_{q_i}).
       \end{equation}
       It follows that
         \begin{equation}\label{eval}
           \begin{split}
             &\underline{\tau}|x_{p_0}-x_{q_I}|_1+\delta_{\ref{lem2}}\sum^{I}_{i=0}|x_{q_i}-x_{p_i}|_1-2(\underline{\tau}+\delta)(\log{N})^{\frac{1}{4d^2}}\\
             &\leq \sum^{I}_{i=0}(\underline{\tau}+\delta_{\ref{lem2}})|x_{q_i}-x_{p_i}|_1+\underline{\tau}\sum^{I-1}_{i=0}|x_{p_{i+1}}-x_{q_i}|_1-2(\underline{\tau}+\delta)(\log{N})^{\frac{1}{4d^2}}\\
             &\le \sum^{I}_{i=0}{\rm T}^*(x_{p_i},x_{q_i})+\sum^{I-1}_{i=0}{\rm T}^*(x_{p_{i+1}},x_{q_i})={\rm T}^*(x_{p_0},x_{q_I}) \\
             &\le (\underline{\tau}+c^2)|x_{p_0}-x_{q_I}|_1+8dMn_1,
       \end{split}
         \end{equation}
         where  we have used \eqref{diss} in the first inequality, \eqref{kkk1} and  \eqref{kkk2} in the second inequality and \eqref{ato de tukau} in the third inequality. Comparing the left and right hand side of \eqref{eval}, we have
         \ben{\label{benki}
           \delta_{\ref{lem2}}\sum^{I}_{i=0}|x_{q_i}-x_{p_i}|_1\leq 2(\underline{\tau}+\delta)(\log{N})^{\frac{1}{4d^2}}+c^2|x_{p_0}-x_{q_I}|_1+8dMn_1.
           }
         Since $|x_{p_0}-x_{q_I}|_1\le 3dn$, for $c>0$ small enough, we have
         \ben{\label{benki2}
           2(\underline{\tau}+\delta)(\log{N})^{\frac{1}{4d^2}}+c^2|x_{p_0}-x_{q_I}|_1\leq 8dMn_1.
           }
Combining \eqref{benki}, \eqref{benki2} and Step 1, we obtain
         \begin{equation}\label{upper-es}
        \frac{(I-1)\delta_{\ref{lem2}}n_1}{4}\leq    \delta_{\ref{lem2}} \sum^{I}_{i=0}|x_{q_i}-x_{p_i}|_1\leq 16dM n_1.
           \end{equation}
   Hence, if we take $M>100d(1+\delta_{\ref{lem2}}^{-1})$, then we have $I\leq M^2$. If $\pi^*$ does not pass through any edge of $\tilde{\rm E}^j(l;n)$, by $\max_{a\in {\rm C}^j(l;n)} \sharp W_a\leq 2dn_1$, then for any $i\in \Iintv{0,I-1}$,  $|x_{p_{i+1}}-x_{q_i}|_1\le 2dn_1$. This yields that
           \begin{equation}\label{eval3}
             \sum^{I-1}_{i=0}|x_{p_{i+1}}-x_{q_i}|_1\leq 2dn_1 I\leq 2d M^3 n_1.
             \end{equation}
          Recall we have defined $n=\lf s{\rm f}_{d,r}(N)\rf\text{ and }n_1=2\lf s^2{\rm f}_{d,r}(N)\rf$.   Therefore, we have
   \al{
             \frac{\delta_{\ref{lem2}}{}n}{4M}&\leq |x_{p_0}-x_{q_I}|_1\\
             &\leq \sum^{I}_{i=0}|x_{q_i}-x_{p_i}|_1+\sum^{I-1}_{i=0}|x_{p_{i+1}}-x_{q_i}|_1\\
             &\leq 16 \delta_{\ref{lem2}}^{-1}d Mn_1+2d M^3 n_1\leq M^4 s n,
   }
   where we have used Step 2 in the first inequality,  \eqref{diss} in the second inequality, \eqref{upper-es} and \eqref{eval3} in the third inequality.
           This leads to a contradiction if $s$ is sufficiently small depending on $M$ and $\delta_{\ref{lem2}}$. Thus we complete the proof.
    \end{proof}
    Proposition~\ref{crucial-prop3}   implies that on the event $A$, ${\rm B}^j(l;n)$ is good for $\tau^*$. Hence, we obtain \eqref{crucial}.
    Together with Proposition~\ref{prob A3}, as in \eqref{combining dayo}, we conclude Proposition~\ref{crucial-don}.
\section*{Acknowledgements}
I would like to thank Professor Ryoki Fukushim giving me a plenty of helpful advice, encouraging me to carry out the research. I also would like to express my gratitude to Professor Antonio Auffinger for suggesting me this problem and introducing the idea of the arguments in \cite{BK93}. I would like to show my appreciation to Professor Takashi Kumagai for a lot of advice and support.
 I am indebted to an anonymous reviewer of an earlier paper for
providing insightful comments.
This research is partially supported by JSPS KAKENHI	16J04042 and Kyoto University Top Global University Project.

\end{document}